%% file: PNCGPaperArxiv.tex
\documentclass[12pt]{article}
\usepackage{color}
\usepackage{epsfig}
\usepackage[subrefformat=parens]{subfig}
\usepackage{graphicx}
\usepackage{amsmath}
\usepackage{amsthm}
\usepackage{amsfonts}
\usepackage{stmaryrd}
\usepackage[mathscr]{eucal}
\usepackage{multirow}
\usepackage[ruled,vlined]{algorithm2e}
\usepackage[justification=centering]{caption}
\usepackage[top=0.5in, bottom=1in, left=1in, right=1.25in]{geometry}
\usepackage{url}

\newtheorem{thm}{Theorem}

\newcommand{\beq}{\begin{equation}}
\newcommand{\eeq}{\end{equation}}
\newcommand{\bfx}{\mathbf{x}}
\newcommand{\bfp}{\mathbf{p}}
\newcommand{\bfg}{\mathbf{g}}
\newcommand{\bfr}{\mathbf{r}}
\newcommand{\bfb}{\mathbf{b}}

\newcommand{\bfA}{\mathbf{A}}
\newcommand{\bfC}{\mathbf{C}}
\newcommand{\bfP}{\mathbf{P}}
\newcommand{\bfM}{\mathbf{M}}

\title{A Nonlinearly Preconditioned Conjugate Gradient Algorithm for Rank-$R$ Canonical Tensor Approximation}

\author{Hans De Sterck\footnotemark[2]\ \footnotemark[3]
\and Manda Winlaw\footnotemark[2]\ \footnotemark[3]}

\date{}

\begin{document}

\maketitle

\renewcommand{\thefootnote}{\fnsymbol{footnote}}

\footnotetext[2]{Department of Applied Mathematics, University of Waterloo, Waterloo, Ontario, N2L 3G1,
Canada (hdesterck@uwaterloo.ca, mwinlaw@uwaterloo.ca)}
\footnotetext[3]{This work was supported by NSERC of Canada and was supported in part by the Scalable Graph Factorization LDRD Project, 13-ERD-072, under the auspices of the U.S. Department of Energy by Lawrence Livermore National Laboratory under Contract DE-AC52-07NA27344.}

\renewcommand{\thefootnote}{\arabic{footnote}}

\begin{abstract}
Alternating least squares (ALS) is often considered the workhorse algorithm for computing the rank-$R$ canonical tensor approximation, but for certain problems its convergence can be very slow. The nonlinear conjugate gradient (NCG) method was recently proposed as an alternative to ALS, but the results indicated that NCG is usually not faster than ALS. To improve the convergence speed of NCG, we consider a nonlinearly preconditioned nonlinear conjugate gradient (PNCG) algorithm for computing the rank-$R$ canonical tensor decomposition. Our approach uses ALS as a nonlinear preconditioner in the NCG algorithm. Alternatively, NCG can be viewed as an acceleration process for ALS. We demonstrate numerically that the convergence acceleration mechanism in PNCG often leads to important pay-offs for difficult tensor decomposition problems, with convergence that is significantly faster and more robust than for the stand-alone NCG or ALS algorithms.  We consider several approaches for incorporating the nonlinear preconditioner into the NCG algorithm that have been described in the literature previously and have met with success in certain application areas. However, it appears that the nonlinearly preconditioned NCG approach has received relatively little attention in the broader community and remains underexplored both theoretically and experimentally. Thus, this paper serves several additional functions, by providing in one place a concise overview of several PNCG variants and their properties that have only been described in a few places scattered throughout the literature, by systematically comparing the performance of these PNCG variants for the tensor decomposition problem, and by drawing further attention to the usefulness of nonlinearly preconditioned NCG as a general tool. In addition, we briefly discuss the convergence of the PNCG algorithm. In particular, we obtain a new convergence result for one of the PNCG variants under suitable conditions, building on known convergence results for non-preconditioned NCG.

\bigskip
\noindent {\bf Keywords. }canonical tensor decomposition, alternating least squares, nonlinear conjugate gradient method, nonlinear preconditioning, nonlinear optimization
\end{abstract}


\section{Introduction}

In this paper, we consider a nonlinearly preconditioned nonlinear conjugate gradient (PNCG) algorithm for computing a canonical rank-$R$ tensor approximation using the Frobenius norm as a distance metric.  The current workhorse algorithm for computing the canonical tensor decomposition is the alternating least squares (ALS) algorithm \cite{CarrollChang1970,Harshman1970,KoldaBader2009}. The ALS method is simple to understand and implement, but for certain problems its convergence can be very slow \cite{TomasiBro2006,KoldaBader2009}. In \cite{ADK2011}, the nonlinear conjugate gradient (NCG) method is considered as an alternative to ALS for solving canonical tensor decomposition problems. However, \cite{ADK2011} found that NCG is usually not faster than ALS. In this paper, we show how incorporating ALS as a nonlinear preconditioner into the NCG algorithm (or, equivalently, accelerating ALS by the NCG algorithm) may lead to significant convergence acceleration for difficult canonical tensor decomposition problems.

Our approach is among extensive, recent, research activity on nonlinear preconditioning for nonlinear iterative solvers \cite{FangSaad2009,Elad2010,WalkerNi2011,DeSterck2012,PETSC2013}, including nonlinear GMRES and NCG. This work builds on original contributions dating back as far as the 1960s \cite{Anderson1965,CGO1977,Pulay1980,OW2000}, but much of this early work is not well-known in the broader community and large parts of the landscape remain unexplored experimentally and theoretically \cite{PETSC2013}; the recent paper \cite{PETSC2013} gives a comprehensive overview of the state of the art in nonlinear preconditioning and provides interesting new directions. 

In this paper, we consider nonlinear preconditioning of NCG for the canonical tensor decomposition problem. We consider several approaches for incorporating the nonlinear preconditioner into the NCG algorithm that are described in the literature (see \cite{BartelsDaniel1974,CGO1977,Mittlemann1980,Elad2010,PETSC2013}). Early references to nonlinearly preconditioned NCG include \cite{BartelsDaniel1974} and \cite{CGO1977}.  Both propose the NCG algorithm as a solution method for solving nonlinear elliptic partial differential equations (PDEs) and while both present NCG algorithms that include a possible nonlinear preconditioner, \cite{CGO1977} actually uses a block nonlinear SSOR method as the nonlinear preconditioner in their numerical experiments. Hager and Zhang's survey paper \cite{HagerZhang2006} describes a linearly preconditioned NCG algorithm, but does not discuss general nonlinear preconditioning for NCG. More recent work on nonlinearly preconditioned NCG includes \cite{Elad2010}, which uses parallel coordinate descent as a nonlinear preconditioner for one variant of NCG applied to $L_1-L_2$ optimization in signal and image processing. The recent overview paper \cite{PETSC2013} on nonlinear preconditioning also briefly mentions nonlinearly preconditioned NCG, but discusses a different variant than \cite{BartelsDaniel1974}, \cite{CGO1977}, \cite{Mittlemann1980} and \cite{Elad2010}. In Section \ref{sec:PNCG}, the differences between the PNCG variants of \cite{BartelsDaniel1974,CGO1977,Mittlemann1980,Elad2010,PETSC2013} will be explained. In Section \ref{sec:PNCG} we will also prove a new convergence result for one of the PNCG variants, building on known convergence results for non-preconditioned NCG. In Section 4, extensive numerical tests on a set of standard test tensors systematically compare the performance of the PNCG variants using ALS as the nonlinear preconditioner, and demonstrate the effectiveness of the overall approach.

As mentioned above, we apply the PNCG algorithm to the tensor decomposition problem which can be described as follows.  Let $\mathscr{X} \in \mathbb{R}^{I_1 \times I_2 \times \ldots \times I_N}$ be a $N$-way or $N$th-order tensor of size $I_1 \times I_2 \times \ldots \times I_N$.  Let $\mathscr{A}_R \in \mathbb{R}^{I_1 \times I_2 \times \ldots \times I_N}$ be a canonical rank-$R$ tensor given by
\beq
\mathscr{A}_R = \sum_{r=1}^R \mathbf{a}_r^{(1)} \circ \ldots \circ \mathbf{a}_r^{(N)} = \llbracket \mathbf{A}^{(1)},\ldots,\mathbf{A}^{(N)}\rrbracket.
\eeq
The canonical tensor $\mathscr{A}_R$ is the sum of $R$ rank-one tensors, with the $r$th rank-one tensor composed of the outer product of $N$ column vectors $\mathbf{a}_r^{(n)} \in \mathbb{R}^{I_n}, \; n = 1, \ldots, N$.  We are interested in finding $\mathscr{A}_R$ as an approximation to $\mathscr{X}$ by minimizing the following function: 
\beq
\label{Eq:MinFcn1}
f(\mathscr{A}_R) = \frac{1}{2}\|\mathscr{X}-\mathscr{A}_R\|^2_F,
\eeq
where $\|\cdot\|_F$ denotes the Frobenius norm of the $N$-dimensional array.

The decomposition of $\mathscr{X}$ into $\mathscr{A}_R$ is known as the canonical tensor decomposition. Popularized by Carroll and Chang \cite{CarrollChang1970} as CANDECOMP and by Harshman \cite{Harshman1970} as PARAFAC in the 1970s, the decomposition is commonly referred to as the CP decomposition where the `C' refers to CANDECOMP and the `P' refers to PARAFAC. The canonical tensor decomposition is commonly used as a data analysis technique in a wide variety of fields including chemometrics, signal processing, neuroscience and web analysis \cite{KoldaBader2009,Acar2008}.  

The ALS algorithm for CP decomposition was first proposed in papers by Carroll and Chang \cite{CarrollChang1970} and Harshman \cite{Harshman1970}. For simplicity we present the algorithm for a three-way tensor $\mathscr{X} \in \mathbb{R}^{I \times J \times K}$.  In this case, the objective function (\ref{Eq:MinFcn1}) simplifies to 
\beq
f(\mathscr{\widehat{X}}) = \frac{1}{2}\|\mathscr{X}-\mathscr{\widehat{X}}\|^2_F \; \mbox{with} \; \mathscr{\widehat{X}} = \sum_{r=1}^k \mathbf{a}_r \circ \bfb_r \circ \mathbf{c}_r  = \llbracket \mathbf{A}, \mathbf{B}, \mathbf{C}\rrbracket.
\eeq
The ALS approach fixes $\mathbf{B}$ and $\mathbf{C}$ to solve for $\mathbf{A}$, then fixes $\mathbf{A}$ and $\mathbf{C}$ to solve for $\mathbf{B}$, then fixes $\mathbf{A}$ and $\mathbf{B}$ to solve for $\mathbf{C}$.  This process continues until some convergence criterion is satisfied.  Once all but one matrix is fixed, the problem reduces to a linear least-squares problem.  Since we are solving a nonlinear equation for a block of variables while holding all the other variables fixed the ALS algorithm is in fact a block nonlinear Gauss-Seidel algorithm.  The algorithm can easily be extended to N-way tensors by fixing all but one of the matrices. The ALS method is simple to understand and implement, but can take many iterations to converge.  It is not guaranteed to converge to a global minimum or even a stationary point of (\ref{Eq:MinFcn1}).  We  can only guarantee that the objective function in (\ref{Eq:MinFcn1}) is nonincreasing at every step of the ALS algorithm. As well, if the ALS algorithm does converge to a stationary point, the stationary point can be heavily dependent on the starting guess.  See \cite{KoldaBader2009, Uschmajew2012} for a discussion on the convergence of the ALS algorithm.

A number of algorithms have been proposed as alternatives to the ALS algorithm.  See \cite{ADK2011, KoldaBader2009, TomasiBro2006, GKT2013} and the references therein for examples. Acar, Dunlavy and Kolda \cite{ADK2011} recently applied a standard NCG algorithm to solve the problem.  They find that NCG is usually not faster than ALS, even though it has its advantages in terms of overfactoring and its ability to solve coupled factorizations \cite{ADK2011,Acar2011}.  In an earlier paper, Paatero \cite{Paatero1999} uses the linear conjugate gradient algorithm to solve the normal equations associated with the CP decomposition and suggests the possible use of a linear preconditioner to increase the convergence speed, however, no extensive numerical testing of the algorithm is performed.  Inspired by the nonlinearly preconditioned nonlinear GMRES method of \cite{DeSterck2012}, we propose in this paper to accelerate the NCG approach of \cite{ADK2011} by considering the use of ALS as a nonlinear preconditioner for NCG.

In terms of notation, throughout the paper we use CG to refer to the linear conjugate gradient algorithm applied to a symmetric positive definite (SPD) linear system without preconditioning, and PCG refers to CG for SPD linear systems with (linear) preconditioning. Similarly, NCG refers to the nonlinear conjugate gradient algorithm for optimization problems without preconditioning, and PNCG refers to the class of (nonlinearly) preconditioned nonlinear conjugate gradient methods for optimization.

The remainder of the paper is structured as follows.  In Section 2, we introduce the standard nonlinear conjugate gradient algorithm for unconstrained continuous optimization. Section 3 gives a concise description of several variants of the PNCG algorithm that we collect from the literature and describe systematically, and it discusses their relation to the PCG algorithm in the linear case, followed by a brief convergence discussion highlighting our new convergence result. In Section 4 we follow the experimental procedure of Tomasi and Bro \cite{TomasiBro2006} to generate test tensors that we use to systematically compare the several PNCG variants we have described with the standard ALS and NCG algorithms. Section 5 concludes.

\section{Nonlinear Conjugate Gradient Algorithm}

The NCG algorithm for continuous optimization is an extension of the CG algorithm for linear systems.  The CG algorithm minimizes the convex quadratic function
\beq
\label{Eq:ConvexFunction}
\phi(\bfx) = \frac{1}{2}\bfx^T\bfA\bfx - \bfb^T\bfx,
\eeq
where $\bfA \in \mathbb{R}^{n \times n}$ is an SPD matrix.  Equivalently, the CG algorithm can be viewed as an iterative method for solving the linear system of equations $\bfA\bfx = \bfb$.  The NCG algorithm is adapted from the CG algorithm and can be applied to any unconstrained optimization problem of the form 
\beq
\min_{\bfx \in \mathbb{R}^n} f(\bfx)
\eeq
where $f:\mathbb{R}^n \rightarrow \mathbb{R}$ is a continuously differentiable function bounded from below.  The general form of the NCG algorithm is summarized in Algorithm \ref{Alg:NCG}.  

\begin{algorithm}[ht]
 \SetAlgoLined
 \KwIn{$\bfx_0$}
 Evaluate $\bfg_0 = \nabla{f}(\bfx_0)$\;
 Set $\bfp_0 \leftarrow -\bfg_0, k \leftarrow 0$\;
 \While{$\bfg_k \neq 0$}{
  Compute $\alpha_k$\;
  $\bfx_{k+1} \leftarrow \bfx_k + \alpha_k\bfp_k$\;
  Evaluate $\bfg_{k+1} = \bfg(x_{k+1})=\nabla f(x_{k+1})$\;
  Compute $\beta_{k+1}$\;
  $\bfp_{k+1} \leftarrow -\bfg_{k+1}+\beta_{k+1}\bfp_k$\;
  $k \leftarrow k +1$\;
   }
 \caption{Nonlinear Conjugate Gradient Algorithm (NCG)}
 \label{Alg:NCG}
\end{algorithm}

The NCG algorithm  is a line search algorithm that generates a sequence of iterates $\bfx_i$, $i \geq 1$ from the initial guess $\bfx_0$ using the recurrence relation
\beq
\label{Eq:NCG:StepUpdate}
\bfx_{k+1} = \bfx_k + \alpha_k \bfp_k.
\eeq
The parameter $\alpha_k > 0$ is the step length and $\bfp_k$ is the search direction generated by the following rule:
\beq
\label{Eq:NCG:SearchDirectionUpdate}
\bfp_{k+1} = -\bfg_{k+1} + \beta_{k+1} \bfp_k, \;\; \bfp_0 = -\bfg_0,
\eeq
where $\beta_{k+1}$ is the update parameter and $\bfg_k = \nabla f(\bfx_k)$ is the gradient of $f$ evaluated at $\bfx_k$.  In the CG algorithm, $\alpha_k$ is defined as
\beq
\alpha_k^{CG} = \frac{\bfr_k^T\bfr_k}{\bfp_k^T\bfA\bfp_k},
\eeq
and $\beta_{k+1}$ is defined as
\beq
\label{Eq:CG:Beta}
\beta_{k+1}^{CG} = \frac{\bfr_{k+1}^T\bfr_{k+1}}{\bfr_{k}^T\bf\bfr_{k}},
\eeq
where $\bfr_k = \nabla \phi(\bfx_k) = \bfA\bfx_k - \bfb$ is the residual. In the nonlinear case $\alpha_k$ is determined by a line search algorithm and $\beta_{k+1}$ can assume various different forms.  We consider three different forms in this paper, given by
\beq
\label{Eq:NCG:BetaFR}
\beta_{k+1}^{FR} = \frac{\bfg_{k+1}^T\bfg_{k+1}}{\bfg_{k}^T\bfg_{k}},
\eeq
\beq
\label{Eq:NCG:BetaPR}
\beta_{k+1}^{PR} = \frac{\bfg_{k+1}^T(\bfg_{k+1}-\bfg_k)}{\bfg_{k}^T\bfg_{k}},
\eeq
\beq
\label{Eq:NCG:BetaHS}
\beta_{k+1}^{HS} = \frac{\bfg_{k+1}^T(\bfg_{k+1}-\bfg_k)}{(\bfg_{k+1}-\bfg_k)^T\bfp_k}.
\eeq
Fletcher and Reeves \cite{FR1964} first showed how to extend the conjugate gradient algorithm to the nonlinear case. By replacing the residual, $\mathbf{r}_k$, with the gradient of the nonlinear objective $f$, they obtained a formula for $\beta_{k+1}$ of the form $\beta_{k+1}^{FR}$.  The variant $\beta_{k+1}^{PR}$ was developed by Polak and Ribi\`{e}re \cite{Polak1969} and the Hestenes-Stiefel \cite{Hestenes1952} formula is given by Equation (\ref{Eq:NCG:BetaHS}).  For all three versions, it can easily be shown that, if a convex quadratic function is optimized using the NCG algorithm and the line search is exact then $\beta_{k+1}^{FR} = \beta_{k+1}^{PR} = \beta_{k+1}^{HS} = \beta_{k+1}^{CG}$ where $\beta_{k+1}^{CG}$ is given by Equation (\ref{Eq:CG:Beta}), see \cite{NocedalWright2006}.  

\section{Preconditioned Nonlinear Conjugate Gradient Algorithm}
\label{sec:PNCG}

In this section we give a concise description of several variants of PNCG that have been proposed in a few places in the literature but have not been discussed and compared systematically in one place, briefly discuss some of their relevant properties, and prove a new convergence property for one of the variants.
Before we introduce PNCG we describe the PCG algorithm for linear systems. We do this because it will be useful for interpreting some of the variants for $\beta_{k+1}$ in the PNCG algorithm. In particular, one variant of the $\beta_{k+1}$ formulas has the property that PNCG applied to the convex quadratic function (\ref{Eq:ConvexFunction}) is equivalent to PCG under certain conditions on the line search and the preconditioner.

\subsection{Linearly Preconditioned Linear Conjugate Gradient Algorithm}

Preconditioning the conjugate gradient algorithm is commonly used in numerical linear algebra to speed up convergence \cite{Saad2003}. The rate of convergence of the linear conjugate gradient algorithm can be bounded by examining the eigenvalues of the matrix $\bfA$ in (\ref{Eq:ConvexFunction}). For example, if the eigenvalues occur in $r$ distinct clusters the CG iterates will approximately solve the problem in $r$ steps \cite{NocedalWright2006}. Thus, one way to improve the convergence of the CG algorithm is to transform the linear system $\bfA\bfx = \bfb$ to improve the eigenvalue distribution of $\bfA$.  Consider a change of variables from $\bfx$ to $\widehat{\bfx}$ via a nonsingular matrix $\bfC$ such that $\widehat{\bfx} = \bfC\bfx$.  This process is known as preconditioning.  The new objective function is
\beq
\widehat{\phi}(\widehat{\bfx}) = \frac{1}{2}\widehat{\bfx}^T(\bfC^{-T} \bfA \bfC^{-1})\widehat{\bfx} - (\bfC^{-T}\bfb)^T\widehat{\bfx},
\eeq
and the new linear system is
\beq
(\bfC^{-T}\bfA\bfC^{-1})\widehat{\bfx} = \bfC^{-T}\bfb.
\label{eq:Plin}
\eeq
Thus, the convergence rate will depend on the eigenvalues of the matrix $\bfC^{-T}\bfA\bfC^{-1}$.  If we choose $\bfC$ such that the condition number of $\bfC^{-T}\bfA\bfC^{-1}$ is smaller than the condition number of $\bfA$ or such that eigenvalues of $\bfC^{-T}\bfA\bfC^{-1}$ are clustered, then hopefully the preconditioned CG algorithm will converge faster than the regular CG algorithm.   The preconditioned conjugate gradient algorithm is given in Algorithm \ref{Alg:PCG}, expressed in terms of the original variable $\bfx$ using the SPD preconditioning matrix $\textbf{P} = \bfC^{-1}\bfC^{-T}$. Note that the preconditioned linear system (\ref{eq:Plin}) can equivalently be expressed as $\textbf{P} \bfA \bfx = \textbf{P} \bfb$, where $\textbf{P} \bfA$ has the same eigenvalues as $\bfC^{-T}\bfA\bfC^{-1}$. 
In Algorithm 2, we do not actually form the matrix $\bfP$ explicitly.  Instead, we solve the linear system $\bfM\mathbf{y}_k = \bfr_k$ for $\textbf{y}_k$ with $\textbf{M} = \bfP^{-1} = \bfC\bfC^T$ and use $\textbf{y}_k$ in place of $\bfP\bfr_k$.  See Algorithm 5.3 in \cite{NocedalWright2006} for the PCG algorithm written in this way. Algorithm \ref{Alg:PCG} is written in terms of $\bfP$ to compare the PCG algorithm with the preconditioned NCG algorithm in what follows.

\begin{algorithm}[ht]
 \SetAlgoLined
 \KwIn{$\bfx_0$, Preconditioner $\bfP = \bfC^{-1}\bfC^{-T}$}
 Evaluate $\bfr_0 = \bfA\bfx_0-\bfb$\;
 Evaluate $\bfp_0 = -\bfP \bfr_0, k \leftarrow 0$\;
 \While{$\bfr_k \neq 0$}{
  $\alpha_k \leftarrow \frac{\bfr_k^T\bfP\bfr_k}{\bfp_k\bfA\bfp_k}$\;
  $\bfx_{k+1} \leftarrow \bfx_k + \alpha_k\bfp_k$\;
  $\bfr_{k+1} \leftarrow \bfr_k + \alpha_k\bfA\bfp_k$\;
  $\beta_{k+1} \leftarrow \frac{\bfr_{k+1}^T\bfP\bfr_{k+1}}{\bfr_k\bfP\bfr_k}$\;
  $\bfp_{k+1} \leftarrow -\bfP\bfr_{k+1}+\beta_{k+1}\bfp_k$\;
  $k \leftarrow k +1$\;
   }
 \caption{Linearly Preconditioned Linear Conjugate Gradient Algorithm (PCG)}
 \label{Alg:PCG}
\end{algorithm}

\subsection{Linearly Preconditioned Nonlinear Conjugate Gradient Algorithm}

We can also apply a linear change of variables, $\widehat{\bfx} = \bfC\bfx$, to the NCG algorithm as is explained in review paper \cite{HagerZhang2006}. The linearly preconditioned NCG algorithm expressed in terms of the original variable $\bfx$ can be described by the following equations:
\begin{align}
\label{Eq:LPNCG:StepUpdate}
\bfx_{k+1} &= \bfx_k + \alpha_k\bfp_k, \\
\label{Eq:LPNCG:SearchDirectionUpdate}
\bfp_{k+1} & = -\bfP\bfg_{k+1} + \breve{\beta}_{k+1}\bfp_k, \;\; \bfp_0 = -\bfP\bfg_0, 
\end{align}
where $\bfP = \bfC^{-1}\bfC^{-T}$.  The formulas for $\breve{\beta}_{k+1}$ remain the same as before (Equations (\ref{Eq:NCG:BetaFR})-(\ref{Eq:NCG:BetaHS})), except that $\bfg_k$ and $\bfp_k$ are replaced by $\bfC^{-T}\bfg_k$ and $\bfC\bfp_k$, respectively.  Thus we obtain linearly preconditioned versions of the $\beta_{k+1}$ parameters of Equations (\ref{Eq:NCG:BetaFR})-(\ref{Eq:NCG:BetaHS}):
\beq
\label{Eq:LPNCG:BetaFR}
\breve{\beta}_{k+1}^{FR} = \frac{\bfg_{k+1}^T\bfP\bfg_{k+1}}{\bfg_{k}^T\bfP\bfg_{k}},
\eeq
\beq
\label{Eq:LPNCG:BetaPR}
\breve{\beta}_{k+1}^{PR} = \frac{\bfg_{k+1}^T\bfP(\bfg_{k+1}-\bfg_k)}{\bfg_{k}^T\bfP\bfg_{k}},
\eeq
\beq
\label{Eq:LPNCG:BetaHS}
\breve{\beta}_{k+1}^{HS} = \frac{\bfg_{k+1}^T\bfP(\bfg_{k+1}-\bfg_k)}{(\bfg_{k+1}-\bfg_k)^T\bfp_{k}}.
\eeq

If we use the linearly preconditioned NCG algorithm with these $\breve{\beta}_{k+1}$ formulas to minimize the convex quadratic function, $\phi(\bfx)$, defined in Equation (\ref{Eq:ConvexFunction}), using an exact line search, where $\bfg_k = \bfr_k$, then the algorithm is the same as the PCG algorithm described in Algorithm \ref{Alg:PCG}.  This can easily be shown in the same way as Equations (\ref{Eq:NCG:BetaFR})-(\ref{Eq:NCG:BetaHS}) are shown to be equivalent to Equation (\ref{Eq:CG:Beta}) in the linear case without preconditioning \cite{HagerZhang2006,NocedalWright2006}. Hager and Zhang's survey paper \cite{HagerZhang2006} describes this linearly preconditioned NCG algorithm, and also notes that $\bfP$ can be chosen differently in every step (see \cite{NazarethNocedal1982}). While a varying $\bfP$ does introduce a certain type of nonlinearity into the preconditioning process, the preconditioning in every step remains a linear transformation, and is thus different from the more general nonlinear preconditioning to be described in the next section, which employs a general nonlinear transformation in every step.

\subsection{Nonlinearly Preconditioned Nonlinear Conjugate Gradient Algorithm}

Suppose instead, we wish to introduce a nonlinear transformation of $\bfx$.  In particular, suppose we consider a nonlinear iterative optimization method such as Gauss-Seidel.  Let $\overline{\bfx}_k$ be the preliminary iterate generated by one step of a nonlinear iterative method, i.e., we write
\beq
\overline{\bfx}_k = P(\bfx_k),
\eeq
which we will use as a nonlinear preconditioner.  Now define the direction generated by the nonlinear preconditioner as 
\beq
\label{Eq:ModifiedGradient}
\overline{\bfg}_k = \bfx_k - \overline{\bfx}_k = \bfx_k - P(\bfx_k).
\eeq
In nonlinearly preconditioned NCG, one considers the nonlinearly preconditioned direction, $\overline{\bfg}_k$, instead of the gradient, $\bfg_k$, in formulating the NCG method \cite{BartelsDaniel1974,CGO1977,Mittlemann1980,Elad2010,PETSC2013}. This idea can be motivated by the linear preconditioning of CG, where $\bfg_k = \bfr_k$ is replaced by the preconditioned gradient $\bfP\bfg_k = \bfP\bfr_k$ in certain parts of Algorithm \ref{Alg:PCG}.
This corresponds to replacing the Krylov space for CG, which is formed by the gradients $\bfg_k = \bfr_k$, with the left-preconditioned Krylov space for PCG, which is formed by the preconditioned gradients $\bfP\bfg_k = \bfP\bfr_k$ \cite{Saad2003}. In a similar way, we replace the nonlinear gradients $\bfg_k$ with the nonlinearly preconditioned directions $\overline{\bfg}_k$. Note that this approach is called nonlinear left-preconditioning in \cite{PETSC2013}, which also considers nonlinear right-preconditioning.

Thus, our nonlinearly preconditioned NCG algorithm is given by the following equations:
\begin{align}
\bfx_{k+1} & = \bfx_k + \alpha_k \bfp_k, \label{Eq:NPNCG:StepUpdate} \\
\bfp_{k+1} & = -\overline{\bfg}_{k+1} + \overline{\beta}_{k+1}\bfp_k, \;\; \bfp_0 = -\overline{\bfg}_0, \label{Eq:NPNCG:SearchDirectionUpdate}
\end{align}
instead of Equations (\ref{Eq:NCG:StepUpdate}) and (\ref{Eq:NCG:SearchDirectionUpdate}) or Equations (\ref{Eq:LPNCG:StepUpdate}) and (\ref{Eq:LPNCG:SearchDirectionUpdate}).  The formulas for $\overline{\beta}_{k+1}$ in Equation (\ref{Eq:NPNCG:SearchDirectionUpdate}) are modified versions of the $\beta_{k+1}$ from Equations (\ref{Eq:NCG:BetaFR})-(\ref{Eq:NCG:BetaHS}) that incorporate $\overline{\bfg}_k$. However, there are several different ways to modify the $\beta_{k+1}$ to incorporate $\overline{\bfg}_k$, leading to several different variants of $\overline{\beta}_{k+1}$. Algorithm \ref{Alg:NPNCG} summarizes the PNCG algorithm, and Table \ref{Table:BetaFormulas} summarizes the variants of $\overline\beta_{k+1}$ we consider in this paper for PNCG.

\begin{algorithm}[ht]
 \SetAlgoLined
 \KwIn{$\bfx_0$}
 Evaluate $\overline{\bfx}_0 = P(\bfx_0)$\;
 Set $\overline{\bfg}_0 = \bfx_0 - \overline{\bfx}_0$\;
 Set $\bfp_0 \leftarrow -\overline{\bfg}_0, k \leftarrow 0$\;
 \While{$\bfg_k \neq 0$}{
  Compute $\alpha_k$\;
  $\bfx_{k+1} \leftarrow \bfx_k + \alpha_k\bfp_k$\;
  $\overline{\bfg}_{k+1} \leftarrow \bfx_{k+1} - P(\bfx_{k+1})$\;
  Compute $\overline{\beta}_{k+1}$\;
  $\bfp_{k+1} \leftarrow -\overline{\bfg}_{k+1}+\overline{\beta}_{k+1}\bfp_k$\;
  $k \leftarrow k +1$\;
   }
 \caption{Nonlinearly Preconditioned Nonlinear Conjugate Gradient Algorithm (PNCG)}
 \label{Alg:NPNCG}
\end{algorithm}

The first set of $\overline{\beta}_{k+1}$ variants we consider are the $\widetilde{\beta}_{k+1}$ shown in column 1 of Table \ref{Table:BetaFormulas}. The $\widetilde{\beta}_{k+1}$ formulas are derived by replacing all occurrences of $\bfg_k$ with $\overline{\bfg}_k$ in the formulas for $\beta_{k+1}$, Equations (\ref{Eq:NCG:BetaFR})-(\ref{Eq:NCG:BetaHS}):
\beq
\label{Eq:NPNCG:BetaFR1}
\widetilde{\beta}_{k+1}^{FR} = \frac{\overline{\bfg}_{k+1}^T\overline{\bfg}_{k+1}}{\overline{\bfg}_{k}^T\overline{\bfg}_{k}},
\eeq
\beq
\label{Eq:NPNCG:BetaPR1}
\widetilde{\beta}_{k+1}^{PR} = \frac{\overline{\bfg}_{k+1}^T(\overline{\bfg}_{k+1}-\overline{\bfg}_k)}{\overline{\bfg}_{k}^T\overline{\bfg}_{k}},
\eeq
\beq
\label{Eq:NPNCG:BetaHS1}
\widetilde{\beta}_{k+1}^{HS} = \frac{\overline{\bfg}_{k+1}^T(\overline{\bfg}_{k+1}-\overline{\bfg}_k)}{(\overline{\bfg}_{k+1}-\overline{\bfg}_k)^T\bfp_k}.
\eeq
This is a straightforward generalization of the $\beta_{k+1}$ expressions in Equations (\ref{Eq:NCG:BetaFR})-(\ref{Eq:NCG:BetaHS}), and the systematic numerical comparisons to be presented in Section \ref{Sec:Numerics} indicate that these choices lead to efficient PNCG methods. The PR variant of this formula is used in \cite{PETSC2013} in the context of PDE solvers.

However, the reader may note that Equations (\ref{Eq:LPNCG:BetaFR})-(\ref{Eq:LPNCG:BetaHS}) suggest different choices for the $\overline{\beta}_{k+1}$ formulas, variants which reduce to the PCG update formulas in the linear case. Indeed, suppose we apply Algorithm \ref{Alg:NPNCG} to the convex quadratic problem, (\ref{Eq:ConvexFunction}), with an exact line search, using a symmetric stationary linear iterative method such as symmetric Gauss-Seidel or Jacobi as a preconditioner. We begin by writing the stationary iterative method in general form as 
\beq
\overline{\bfx}_k = P(\bfx_k) = \bfx_k - \bfP\bfr_k,
\eeq
where the SPD preconditioning matrix $\bfP$ is often written as $\bfM^{-1}$ and $\bfr_k=\bfg_k$.
The search direction $\overline{\bfg}_k$ from Equation (\ref{Eq:ModifiedGradient}) simply becomes
\beq
\overline{\bfg}_k = \bfx_k - \overline{\bfx}_{k} = \bfP\bfr_k=\bfP\bfg_k.
\eeq
This immediately suggests a generalization of the linearly preconditioned NCG parameters $\breve{\beta}_{k+1}$ of Equations (\ref{Eq:LPNCG:BetaFR})-(\ref{Eq:LPNCG:BetaHS}) to the case of nonlinear preconditioning: replacing all occurrences of $\bfP\bfg_k$ with $\overline{\bfg}_k$ we obtain the expressions
\beq
\label{Eq:NPNCG:BetaFR2}
\widehat{\beta}_{k+1}^{FR} = \frac{{\bfg}_{k+1}^T\overline{\bfg}_{k+1}}{{\bfg}_{k}^T\overline{\bfg}_{k}},
\eeq
\beq
\label{Eq:NPNCG:BetaPR2}
\widehat{\beta}_{k+1}^{PR} = \frac{{\bfg}_{k+1}^T(\overline{\bfg}_{k+1}-\overline{\bfg}_k)}{{\bfg}_{k}^T\overline{\bfg}_{k}},
\eeq
\beq
\label{Eq:NPNCG:BetaHS2}
\widehat{\beta}_{k+1}^{HS} = \frac{{\bfg}_{k+1}^T(\overline{\bfg}_{k+1}-\overline{\bfg}_k)}{({\bfg}_{k+1}-{\bfg}_k)^T\bfp_k}.
\eeq

Expressions of this type have been used in \cite{BartelsDaniel1974,CGO1977,Mittlemann1980,Elad2010}.  It is clear that the PNCG algorithm with this second set of expressions, which are listed in the right column of Table \ref{Table:BetaFormulas}, reduces to the PCG algorithm in the linear case, since the $\widehat{\beta}_{k+1}$ reduce to the $\breve{\beta}_{k+1}$ in the case of a linear preconditioner, and the $\breve{\beta}_{k+1}$ in turn reduce to the $\beta_{k+1}$ from the PCG algorithm when solving an SPD linear system. For completeness, we state this formally in the following theorem.

\begin{thm}
Let $\bfA$ and $\bfP$ be SPD matrices. Then PNCG (Algorithm \ref{Alg:NPNCG}) with $\widehat{\beta}_{k+1}^{FR}$, $\widehat{\beta}_{k+1}^{PR}$ or $\widehat{\beta}_{k+1}^{HS}$ of Table \ref{Table:BetaFormulas} applied to the convex quadratic problem $\phi(\bfx) = \frac{1}{2}\bfx^T\bfA\bfx - \bfb^T\bfx$ using an exact line search and a symmetric linear stationary iterative method with preconditioning matrix $\bfP$ as the preconditioner, reduces to PCG (Algorithm \ref{Alg:PCG}) applied to the linear system $\bfA \bfx=\bfb$ with the same preconditioner.
\end{thm}

Thus, for the nonlinearly preconditioned NCG method, we have two sets of $\overline{\beta}_{k+1}$ formulas: the $\widehat{\beta}_{k+1}$ formulas have the property that the PNCG algorithm reduces to the PCG algorithm in the linear case, whereas the $\widetilde{\beta}_{k+1}$ formulas do not enjoy this property. Due to this property, one might expect the $\widehat{\beta}_{k+1}$ formulas to perform better, but our numerical tests show that this is not necessarily the case. Hence, we will use both the $\widetilde{\beta}_{k+1}$ and $\widehat{\beta}_{k+1}$ formulas in our numerical tests. 

Next we investigate aspects of convergence of the PNCG algorithm.  For the NCG algorithm without preconditioning, global convergence can be proved for the Fletcher-Reeves method applied to a broad class of objective functions, in the sense that
\beq
\liminf_{k \rightarrow \infty} \| \bfg_k \| = 0,
\eeq
when the line search satisfies the strong Wolfe conditions (see \cite{HagerZhang2006,NocedalWright2006} for a general discussion on NCG convergence).  Global convergence cannot be proved in general for the Polak-Ribi\`{e}re or Hestenes-Stiefel variants. Nevertheless, these methods are also widely used and may perform better than Fletcher-Reeves in practice. Global convergence can be proved for variants of these methods in which every search direction $\bfp_k$ is guaranteed to be a descent direction ($\bfg_k^T \bfp_k<0$), and in which the iteration is restarted periodically with a steepest-descent step.

It should come as no surprise that general convergence results for the PNCG algorithm are also difficult to obtain: use of a nonlinear preconditioner only exacerbates the already considerable theoretical difficulties in analyzing the convergence properties of these types of nonlinear optimization methods. However, with the use of the following theorem we will be able to establish global convergence for a restarted version of the $\widehat{\beta}^{FR}_k$ variant of the PNCG algorithm with a line search satisfying the strong Wolfe conditions, under the condition that the nonlinear preconditioner produces descent directions. Since the proof is dependent on the line search satisfying the strong Wolfe conditions we include the conditions for completeness.  The strong Wolfe conditions require the step length parameter, $\alpha_k$, in the update equation, $x_{k+1} = x_k + \alpha_kp_k$, of any line search method to satisfy the following:
\begin{align}
f(x_k + \alpha_kp_k) &\leq f(x_k) + c_1\alpha_k\nabla f_k^Tp_k, \label{Equation:WolfeSDC}\\
|\nabla f(x_k+\alpha_kp_k)^Tp_k | & \leq c_2 | \nabla f_k^Tp_k |, \label{Equation:WolfeCC}
\end{align}
with $0 < c_1 < c_2 < 1$. Condition (\ref{Equation:WolfeSDC}) is known as the \textit{sufficient decrease} or \textit{Armijo condition} and condition (\ref{Equation:WolfeCC}) is known as the \textit{curvature condition}.  The proof of our theorem relies on condition (\ref{Equation:WolfeCC}).  We will use this condition to help show that the PNCG search directions $\bfp_k$ obtained using $\widehat{\beta}^{FR}_k$ are descent directions when the nonlinear preconditioner produces descent directions.  To show this we follow the proof technique of Lemma 5.6 in \cite{NocedalWright2006}.

\begin{thm} 
Consider the PNCG algorithm given in Algorithm \ref{Alg:NPNCG} with $\overline{\beta}_{k+1} = \widehat{\beta}^{FR}_{k+1}$ and where $\alpha_k$ satisfies the strong Wolfe conditions. Let $P(\bfx)$ be a nonlinear preconditioner such that $-\overline{\bfg}(\bfx_k)=P(\bfx_k)-\bfx_k$ is a descent direction for all $k$, i.e., $-\bfg_k^T\overline{\bfg}_k<0$. Suppose the objective function $f$ is bounded below in $\mathbb{R}^n$ and $f$ is continuously differentiable in an open set $\mathcal{N}$ containing the level set $\mathcal{L} := \{\bfx : f(\bfx) \leq f(\bfx_0)\}$, where $\bfx_0$ is the starting point of the iteration.  Assume also that the gradient $\bfg_k$ is Lipschitz continuous on $\mathcal{N}$.  Then,
\beq
\label{Eq:ZC}
\sum_{k\geq 0} \cos^2\theta_k \| \bfg_k\|^2 < \infty,
\eeq
where
\beq
\cos \theta_k =\frac{-\bfg_k^T \bfp_k}{\|\bfg_k\| \|\bfp_k\|}.
\eeq
\label{thm:convergence}
\end{thm}
\begin{proof} 
We show that $\bfp_k$ is a descent direction, i.e., $\bfg_k^T\bfp_k < 0 \; \forall \; k$. Then condition (\ref{Eq:ZC}) follows directly from Theorem 3.2 of Nocedal and Wright \cite{NocedalWright2006} which states that condition (\ref{Eq:ZC}) holds for any iteration of the form $\bfx_{k+1} = \bfx_k + \alpha_k \bfp_k$ provided that the above conditions hold for $\alpha_k$, $f$ and $\bfg_k$, and where $\bfp_k$ is a descent direction.
\newline \newline
Instead of proving that $\bfg_k^T\bfp_k < 0$ directly, we will prove the following:
\beq
\label{Eq:ProofCondition}
-\frac{1}{1-c_2} \leq \frac{\bfg_k^T\bfp_k}{\bfg_k^T\overline{\bfg}_k} \leq \frac{2c_2-1}{1-c_2}, \;\; k \geq 0,
\eeq
where $0 < c_2 < \frac{1}{2}$ is the constant from the curvature condition of the strong Wolfe conditions:
\beq
\label{Eq:WolfeCurvatureCondition}
|\bfg_{k+1}^T\bfp_k| \leq c_2 |\bfg_k^T\bfp_k|.
\eeq
Note, that the function $t(\xi) = (2\xi-1)/(1-\xi)$ is monotonically increasing on the interval $[0,\frac{1}{2}]$ and that $t(0) = -1$ and $t(\frac{1}{2}) = 0$. Thus, because $c_2 \in (0,\frac{1}{2})$, we have
\beq
\label{Eq:UpperBound}
-1 < \frac{2c_2-1}{1-c_2} < 0.
\eeq
Also note that since $-\overline{\bfg}_k$ is a descent direction, $\bfg_k^T\overline{\bfg}_k > 0$.  So, if (\ref{Eq:ProofCondition}) holds then $\bfg_k^T\bfp_k < 0$ and $\bfp_k$ is a descent direction. 

We use an inductive proof to show that (\ref{Eq:ProofCondition}) is true. For $k = 0$, we use the definition of $\bfp_0$ to get,
\beq
\frac{\bfg_0^T\bfp_0}{\bfg_0^T\overline{\bfg}_0} = \frac{-\bfg_0^T\overline{\bfg}_0}{\bfg_0^T\overline{\bfg}_0} = -1.
\eeq
From (\ref{Eq:UpperBound}) we have
\beq
\frac{\bfg_0^T\bfp_0}{\bfg_0^T\overline{\bfg}_0} = -1 \leq \frac{2c_2-1}{1-c_2}.
\eeq
Note, that the function $t(\xi) = -1/(1-\xi)$ is monotonically decreasing on the interval $[0,\frac{1}{2}]$ and that $t(0) = -1$ and $t(\frac{1}{2}) = -2$. Thus, because $c_2 \in (0,\frac{1}{2})$, we have
\beq
\label{Eq:LowerBound}
-2 < -\frac{1}{1-c_2} < -1.
\eeq
Thus,
\beq
\frac{\bfg_0^T\bfp_0}{\bfg_0^T\overline{\bfg}_0} = -1 \geq -\frac{1}{1-c_2}.
\eeq
Now suppose
\beq
-\frac{1}{1-c_2} \leq \frac{\bfg_l^T\bfp_l}{\bfg_l^T\overline{\bfg}_l} \leq \frac{2c_2-1}{1-c_2}, \;\; l = 1,\ldots,k.
\eeq
We need to show that (\ref{Eq:ProofCondition}) is true for $k+1$.  Using the definition of $\bfp_{k+1}$ we have, 
\begin{align}
\frac{\bfg_{k+1}^T\bfp_{k+1}}{\bfg_{k+1}^T\overline{\bfg}_{k+1}} & = \frac{\bfg_{k+1}^T\left(-\overline{\bfg}_{k+1} + \widehat{\beta}_{k+1}^{FR}\bfp_{k}\right)}{\bfg_{k+1}^T\overline{\bfg}_{k+1}} \notag \\
& = -1 + \widehat{\beta}_{k+1}^{FR}\frac{\bfg_{k+1}^T\bfp_{k}}{\bfg_{k+1}^T\overline{\bfg}_{k+1}}.
\label{Eq:PKDef} 
\end{align}
From the Wolfe condition, Equation (\ref{Eq:WolfeCurvatureCondition}), and the inductive hypothesis, which implies that $\bfg_k^T\bfp_k < 0$, we can write
\beq
c_2 \bfg_k^T\bfp_k \leq \bfg_{k+1}^T\bfp_k \leq - c_2 \bfg_k^T\bfp_k.
\eeq
Combining this with Equation (\ref{Eq:PKDef}), we have
\beq
-1 + c_2\widehat{\beta}_{k+1}^{FR}\frac{\bfg_{k}^T\bfp_{k}}{\bfg_{k+1}^T\overline{\bfg}_{k+1}} \leq \frac{\bfg_{k+1}^T\bfp_{k+1}}{\bfg_{k+1}^T\overline{\bfg}_{k+1}} \leq -1 - c_2\widehat{\beta}_{k+1}^{FR}\frac{\bfg_{k}^T\bfp_{k}}{\bfg_{k+1}^T\overline{\bfg}_{k+1}}
\eeq
So,
\begin{align*}
\frac{\bfg_{k+1}^T\bfp_{k+1}}{\bfg_{k+1}^T\overline{\bfg}_{k+1}} & \geq -1 + c_2\widehat{\beta}_{k+1}^{FR}\frac{\bfg_{k}^T\bfp_{k}}{\bfg_{k+1}^T\overline{\bfg}_{k+1}} \\
& = -1 + c_2\left(\frac{\bfg_{k+1}^T\overline{\bfg}_{k+1}}{\bfg_k^T\overline{\bfg}_k}\right)\frac{\bfg_{k}^T\bfp_{k}}{\bfg_{k+1}^T\overline{\bfg}_{k+1}} \\
& = -1 + c_2\left(\frac{\bfg_{k}^T\bfp_{k}}{\bfg_k^T\overline{\bfg}_k}\right)\\
& \geq -1 - \frac{c_2}{1-c_2}\\
& = -\frac{1}{1-c_2},
\end{align*}
and
\begin{align*}
\frac{\bfg_{k+1}^T\bfp_{k+1}}{\bfg_{k+1}^T\overline{\bfg}_{k+1}} & \leq -1 - c_2\widehat{\beta}_{k+1}^{FR}\frac{\bfg_{k}^T\bfp_{k}}{\bfg_{k+1}^T\overline{\bfg}_{k+1}} \\
& = -1 - c_2\left(\frac{\bfg_{k+1}^T\overline{\bfg}_{k+1}}{\bfg_k^T\overline{\bfg}_k}\right)\frac{\bfg_{k}^T\bfp_{k}}{\bfg_{k+1}^T\overline{\bfg}_{k+1}} \\
& = -1 - c_2\left(\frac{\bfg_{k}^T\bfp_{k}}{\bfg_k^T\overline{\bfg}_k}\right)\\
& \leq -1 + \frac{c_2}{1-c_2}\\
& = \frac{2c_2-1}{1-c_2}.
\end{align*}
\end{proof}

We can now easily establish that convergence holds for a restarted version of the PNCG algorithm with $\widehat{\beta}^{FR}_{k+1}$ if a nonlinear preconditioner is used that produces descent directions: If we use the steepest decent direction as the search direction on every $m$th iteration of the algorithm and then restart the PNCG algorithm with $\bfp_{m+1} = -\overline{\bfg}_{m+1} = -\bfx_m+P(\bfx_m)$, then Equation (\ref{Eq:ZC}) of Theorem \ref{thm:convergence} is still satisfied for the combined process and
\beq
\liminf_{k \rightarrow \infty} \| \bfg_k \| = 0,
\eeq
since $\cos\theta_k = 1$ for the steepest descent steps \cite{NocedalWright2006}. Thus we are guaranteed overall global convergence for this method. Note that the proof for global convergence of NCG using $\beta^{FR}_{k+1}$ without restarting (Theorem 5.7 in \cite{NocedalWright2006}) does not carry over to the case of unrestarted PNCG with $\widehat{\beta}^{FR}_{k+1}$.  

For (\ref{Eq:ZC}) to hold we must assume that $\beta_{k+1}=\widehat{\beta}^{FR}_{k+1}$.  However, if we use a more restrictive line search we can show that (\ref{Eq:ZC}) holds for any variant of $\beta_{k+1}$ provided that the remaining assumptions of Theorem 2 hold.  Suppose we use an ``ideal'' line search at every step of the PNCG algorithm where a line search is considered ideal if $\alpha_k$ is a stationary point of $f(\bfx_k+\alpha_k\bfp_k)$. If $\alpha_k$ is a stationary point of $f(\bfx_k+\alpha_k\bfp_k)$ then $\nabla f(\bfx_k+\alpha_k\bfp_k)^T\bfp_k = \bfg_{k+1}^T\bfp_k = 0$ and using the definition of $\bfp_k$ we have
\begin{align*}
\bfg_k^T\bfp_k & = \bfg_{k}^T\left(-\overline{\bfg}_{k} + {\beta}_{k}\bfp_{k-1}\right) \\
& = -\bfg_{k}^T\overline{\bfg}_{k} + {\beta}_{k}\bfg_{k}^T\bfp_{k-1} \\
& = -\bfg_{k}^T\overline{\bfg}_{k} < 0,
\end{align*}
provided $-\overline{\bfg}_{k}$ is a descent direction.  Thus, $\bfp_k$ is a descent direction for all $k$ and (\ref{Eq:ZC}) holds for all variants of $\beta_{k+1}$.  This implies that any restarted version of the PNCG algorithm with an ideal line search is guaranteed to converge provided $-\overline{\bfg}_{k}$ is a descent direction. See \cite{CGO1977} for a similar proof when $\overline{\bfg}_{k} = \textbf{M}^{-1}\bfg_k$ and $\textbf{M}$ is a positive definite matrix, which guarantees that $-\overline{\bfg}_{k}$ is a descent direction.  We should also note that performing an ideal line search at every step of the PNCG algorithm is often prohibitively expensive and thus not used in practice.

Both convergence results require that the nonlinearly preconditioned directions $-\overline{\bfg}_k=P(\bfx_k)-\bfx_k$ be descent directions. If one assumes a continuous preconditioning function $P(\bfx)$ such that $-\overline{\bfg}(\bfx)=P(\bfx)-\bfx$ is a descent direction for all $\bfx$ in a neighbourhood of an isolated local minimizer $\bfx^*$ of a continuously differentiable objective function $f(\bfx)$, then this implies that the nonlinear preconditioner satisfies the fixed-point condition $\bfx^*=P(\bfx^*)$, which is a natural condition for a nonlinear preconditioner. It is often the case in nonlinear optimization that convergence results only hold under restrictive conditions and are mainly of theoretical value.  In practice, numerical results may show satisfactory convergence behaviour for much broader classes of problems.  Our numerical results will show that this is also the case for PNCG applied to canonical tensor decomposition: While the ALS preconditioner satisfies the fixed-point property, it is not guaranteed to produce descent directions. Nevertheless, convergence was generally observed numerically for all the PNCG variants we considered, with the $\widetilde{\beta}_{PR}$ variant producing the fastest results in most cases.

\subsection{Application of the PNCG Algorithm to the CP Optimization Problem}

Thus far we have described the PNCG algorithm in very general terms.  The algorithm can be applied to any continuously differentiable function bounded from below using any nonlinear iterative method as a preconditioner.  We now discuss how to apply Algorithm \ref{Alg:NPNCG} to the CP optimization problem.  The two quantities that are most important in the computation of Algorithm \ref{Alg:NPNCG} are the gradient, $\bfg_k$, and the preconditioned value, $P(\bfx_k)$. Not only is the gradient used in some of the formulas for $\beta_{k+1}$, it is also used in calculating the step length parameter, $\alpha_k$.  We choose to use the ALS algorithm as a preconditioner since it is the standard algorithm used to solve the CP decomposition problem.  We briefly revisit the ALS algorithm before discussing the computation of the gradient, $\bfg_k$. The CP optimization problem for an $N$-way tensor $\mathscr{X} \in \mathbb{R}^{I_1 \times \cdots \times I_N}$ is given by
\beq
\min f(\bfA^{(1)}, \ldots, \bfA^{(N)}) = \frac{1}{2}\| \mathscr{X} - \llbracket \bfA^{(1)}, \ldots, \bfA^{(N)}\rrbracket \|^2_F,
\label{Eq:CPMinProblem}
\eeq
where $\bfA^{(n)}, \; n = 1,\ldots, N$, is a factor matrix of size $I_n \times R$ and the following size parameters are defined:
\beq
K = \prod_{l=1}^N I_l, \;\;\;\; \overline{K}^{(n)}  = \prod^N_{l=1, l \neq n} I_l.
\eeq 
Rather than solve (\ref{Eq:CPMinProblem}) for $\bfA^{(1)}$ through $\bfA^{(N)}$ simultaneously, the ALS algorithm solves for each factor matrix one at a time.  The exact solution for each factor matrix is given by 
\beq
\bfA^{(n)} = \mathbf{X}_{(n)}\bfA^{(-n)}(\mathbf{\Gamma}^{(n)})^{\dagger},
\label{Eq:ALSSimplified}
\eeq
where
\beq
\mathbf{\Gamma}^{(n)} = \mathbf{\Upsilon}^{(1)}  * \cdots * \mathbf{\Upsilon}^{(n-1)}  * \mathbf{\Upsilon}^{(n+1)}  * \cdots * \mathbf{\Upsilon}^{(N)},
\eeq
\beq
\mathbf{\Upsilon}^{(n)} = \bfA^{(n)T} \bfA^{(n)},
\eeq
\beq
\bfA^{(-n)} = \bfA^{(N)} \odot \cdots  \odot \bfA^{(n+1)}  \odot \bfA^{(n-1)} \odot \cdots  \odot\bfA^{(1)}, 
\eeq
where $\odot$ is the Khatri-Rho product \cite{KoldaBader2009}, $*$ denotes the elementwise product, and $\mathbf{X}_{(n)} \in \mathbb{R}^{I_n \times \overline{K}^{(n)}}$ is the mode-$n$ matricization of $\mathscr{X}$, obtained by stacking the $n$-mode fibers of $\mathscr{X}$ in its columns in a regular way as defined in \cite{KoldaBader2009}.  For more details of the derivation of Equation (\ref{Eq:ALSSimplified}) see \cite{KoldaBader2009}.

The primary cost of solving for $\bfA^{(n)}$ is multiplying the matricized tensor, $\mathbf{X}_{(n)}$,  with the Khatri-Rao product, $\bfA^{(-n)}$.  The matrix $\mathbf{X}_{(n)}$ is of size $I_n \times \overline{K}^{(n)}$  and $\bfA^{(-n)}$ is of size $\overline{K}^{(n)} \times R$ where $\overline{K}^{(n)} = K/I_n$. Thus the cost of computing Equation (\ref{Eq:ALSSimplified}), measured in terms of the number of operations, is $O(KR)$.  One iteration of the ALS algorithm requires us to solve for each factor matrix, thus each iteration of the ALS algorithm has a computational cost of $O(NKR)$.   

We now discuss the gradient of the objective function in (\ref{Eq:CPMinProblem}).  It can be written as a vector of matrices
\beq
\nabla f(\mathscr{A}_R) = \mathbf{G}(\mathscr{A}_R) = (\mathbf{G}^{(1)},\ldots, \mathbf{G}^{(n)}),
\eeq
where $\mathbf{G}^{(n)} \in \mathbb{R}^{I_n \times R}, \; n = 1,\ldots, N$.  Each matrix $\mathbf{G}^{(n)}, \; n = 1,\ldots, N$, is given by
\beq
\mathbf{G}^{(n)} = -\mathbf{X}_{(n)}\bfA^{(-n)} + \bfA^{(n)}\mathbf{\Gamma}^{(n)}.
\label{Eq:Gradient}
\eeq
The derivation of (\ref{Eq:Gradient}) can be found in \cite{ADK2011}.  From Equation (\ref{Eq:Gradient}) we can see that the cost of computing one gradient matrix, $\mathbf{G}^{(n)}$, is dominated by the calculation of $\mathbf{X}_{(n)}\bfA^{(-n)}$ and thus the computational cost of computing the gradient, $\nabla f(\mathscr{A}_R)$, is $O(NKR)$, the same as one iteration of the ALS algorithm.  

\section{Numerical Results}
\label{Sec:Numerics}

To test our PNCG algorithm we randomly generate artificial tensors of different sizes, ranks, collinearity, and heteroskedastic and homoskedastic noise, which constitute standard test problems for CP decomposition \cite{TomasiBro2006}. We then compare the performance of the CP factorization using the PNCG algorithm with results from using the ALS and NCG algorithms.

\subsection{Problem Description}

The artificial tensors are generated using the methodology of \cite{TomasiBro2006}.  All the tensors we consider are 3-way tensors.  Each dimension has the same size but we consider tensors of three different sizes, $I = 20$, 50 and 100.  The factor matrices $\bfA^{(1)}$, $\bfA^{(2)}$, and $\bfA^{(3)}$ are generated randomly so that the collinearity of the factors in each mode is set to a particular level $C$.  The steps necessary to create the factor matrices are outlined in \cite{TomasiBro2006}.  Thus,
\beq
\frac{\mathbf{a}_r^{(n)T}\mathbf{a}_s^{(n)}}{\|\mathbf{a}_r^{(n)}\|\|\mathbf{a}_s^{(n)}\|} = C,
\eeq
for $r \neq s,\;r,s=1,\ldots,R$ and $n = 1,2,3$. As in \cite{ADK2011}, the values of $C$ we consider are $0.5$ and $0.9$, where higher values of $C$ make the problem more difficult.  We consider two different values for the rank, $R = 3$ and $R = 5$.   For each combination of $R$ and $C$ we generate a set of factor matrices.  Once we have converted these factors into a tensor and added noise our goal is to recover these underlying factors using the different optimization algorithms.  From a given set of factor matrices we are able to generate nine test tensors by adding different levels of homoskedastic and heteroskedastic noise.  Homoskedastic noise refers to noise with constant variance whereas heteroskedastic noise refers to noise with differing variance.  The noise ratios we consider for homoskedastic and heteroskedastic noise are $l_1 = 1, 5, 10$ and $l_2 = 0, 1, 5$, respectively, see \cite{TomasiBro2006,ADK2011}.  Suppose $\mathscr{N}_1$,$\mathscr{N}_2 \in \mathbb{R}^{I_1 \times \cdots \times I_N}$ are random tensors with entries chosen from a standard normal distribution.  Then we generate the test tensors as follows.  Let the original tensor be
\beq
\mathscr{X} =  \llbracket \mathbf{A}^{(1)}, \mathbf{A}^{(2)}, \mathbf{A}^{(3)} \rrbracket.
\eeq
Homoskedastic noise is added to give:
\beq
\mathscr{X}' = \mathscr{X} + ({100}/{l_1} - 1)^{-\frac{1}{2}}\frac{\|\mathscr{X}\|}{\|\mathscr{N}_1\|}\mathscr{N}_1,
\eeq
and then heteroskedastic noise is added to give:
\beq
\mathscr{X}'' = \mathscr{X}' + ({100}/{\mathit{l}_2} - 1)^{-\frac{1}{2}}\frac{\|\mathscr{X}'\|}{\|\mathscr{N}_2 *\mathscr{X}'\|}\mathscr{N}_2 *\mathscr{X}'.
\eeq
The optimization algorithms are applied to the test tensor $\mathscr{X}''$ and in the case where $l_2 = 0$, $\mathscr{X}'' = \mathscr{X}'$. To test the performance of each optimization algorithm we apply the algorithm to the test tensor $\mathscr{X}''$ using 20 different random starting values, where the same 20 starting values are used for each algorithm. Thus for each size, $I = 20, 50$ and $100$ we generate 36 test tensors since we consider 2 different ranks, 2 different collinearity values, 1 set of factor matrices for each combination of $C$ and $R$ and 9 different levels of noise and for each of these test tensors we apply a given optimization algorithm 20 different times using different random starting values.

\subsection{Results}

We begin by presenting numerical results for the smallest case where $I = 20$. All numerical experiments where performed on a Linux Workstation with a Quad-Core Intel Xeon 3.16GHz processor and 8GB RAM. We use the NCG algorithm from the Poblano toolbox for MATLAB \cite{Poblano} which uses the Mor\`{e}-Thuente line search algorithm.  We use the same line search algorithm for the PNCG algorithm.  The line search parameters are as follows:  $10^{-4}$ for the sufficient decrease condition tolerance, $10^{-2}$ for the curvature condition tolerance, an initial step length of 1 and a maximum of 20 iterations.  The ALS algorithm we use is from the tensor toolbox for MATLAB \cite{TensorToolbox}; however, we use a different normalization of the factors (as explained below) and we use the gradient norm as a stopping condition instead of the relative function change.  In the CP decomposition, it is often useful to assume that the columns of the factor matrices, $\bfA^{(n)}$, are normalized to length one with the weights absorbed into a vector $\boldsymbol\lambda \in \mathbb{R}^k$.  Thus
\beq
\mathscr{X} \approx \sum_{r=1}^k \lambda_r a_r^{(1)} \circ \ldots \circ a_r^{(N)}.
\eeq
In our ALS algorithm the factors are normalized such that $\boldsymbol\lambda$ is distributed evenly over all the factors.  Also note that, while the gradient norm is used as a stopping condition for the ALS algorithm, the calculation of the gradient is not included in the timing results for the ALS algorithm.  For all three algorithms, ALS, NCG and PNCG, there are three stopping conditions; all are set to the same value for each algorithm.  They are as follows: $10^{-9}$ for the gradient norm divided by the number of variables, $\|\textbf{G}(\mathscr{A}_R)\|_2/N$ where $N$ is the number of variables in $\mathscr{X}$, $10^4$ for the maximum number of iterations and $10^5$ for the maximum number of function evaluations.   

For $I = 20$ and $R=3$, Table \ref{Table:SummarizeR3} summarizes the results for each algorithm while Table \ref{Table:SummarizeR5} summarizes the results for $I = 20$ and $R = 5$.    For each value of the rank, $R$, there are two possible values for the collinearity, $C = 0.5$ and $C = 0.9$.   Once the collinearity has been fixed, a test tensor is created and there are nine different combinations of homoskedastic and heteroskedastic noise added to each test tensor.  We then generate 20 different initial guesses with components chosen randomly from a uniform distribution between 0 and 1.  Each algorithm is tested using each of these initial guesses. Thus, for each collinearity value there are 180 CP decompositions performed by each algorithm.  Each table reports the overall timings for the 180 CP decompositions. The timing is written in the form $a \pm b$ where $a$ is the mean time and $b$ is the standard deviation.  The first number in brackets represents the number of CP decompositions that converge before reaching the maximum number of iterations or function evaluations out of a possible 180.  All timing calculations are performed for the converged runs only.  The second number in brackets represents the number of runs where the algorithm is able to recover the original set of factor matrices.  We use a measure, defined in \cite{TomasiBro2006}, known as congruence to determine if an algorithm is able to recover the original factors where the congruence between two rank-one tensors, $\mathscr{X} = \mathbf{a} \circ \mathbf{b} \circ \mathbf{c}$ and $\mathscr{Y} = \mathbf{p} \circ \mathbf{q} \circ \mathbf{r}$ is defined as 
\beq
\mbox{cong}(\mathscr{X},\mathscr{Y}) = \frac{|\mathbf{a}^T\mathbf{p}|}{\|\mathbf{a}\|\|\mathbf{p}\|} \cdot \frac{|\mathbf{b}^T\mathbf{q}|}{\|\mathbf{b}\|\|\mathbf{q}\|} \cdot \frac{|\mathbf{c}^T\mathbf{r}|}{\|\mathbf{c}\|\|\mathbf{r}\|}.
\eeq
If the congruence is above 0.97 ($\approx 0.99^3$) for every component rank-one tensor then we say that the algorithm has successfully recovered the original factor matrices.  Since the CP decomposition is unique up to a permutation of the component rank-one tensors, we consider all permutations when calculating congruences and choose the permutation that results in the greatest sum of congruences of rank-one tensors. We also calculate the congruences for all runs regardless of whether or not they converge.  

From the results it is clear that when $C = 0.5$, ALS is the fastest algorithm.  The results also indicate that for a given formula for $\beta$, NCG is faster that either PNCG algorithm.  However, when the collinearity is 0.5, it is known that the problem is relatively easy \cite{TomasiBro2006,ADK2011,DeSterck2012}, so we don't necessarily expect the preconditioned algorithm to outperform the standard algorithm, and the additional time needed to perform the preconditioning may actually slow the algorithm down relative to the original algorithm.  The results change when we look at the more difficult problem, $C=0.9$.   In this case, the PNCG algorithm with the $\widetilde{\beta}^{PR}$ variant is the fastest. The ALS algorithm is the slowest and for a given formula for $\beta$, both PNCG algorithms are faster than the NCG algorithm by a factor between 2 and 4: nonlinear preconditioning significantly speeds up the NCG algorithm.  The one exception is for $R = 3$ where the PNCG algorithm with $\widetilde{\beta}^{HS}$ is slower than the NCG algorithm.  However, in this case the number of convergent runs for $\widetilde{\beta}^{HS}$ is $100\%$ while only $92.78\%$ of the runs for $\beta^{HS}$ are convergent.  We also see from Tables \ref{Table:SummarizeR3} and \ref{Table:SummarizeR5} that the number of times each algorithm is able to recover the original factor matrices successfully is approximately the same for each algorithm for a given combination of $R$ and $C$.  In the case where $R = 5$ and $C = 0.9$, this number is quite low, however, these results closely match the results found in \cite{Acar2011} and we note that all of the successful runs occur when there is very little noise ($\mathit{l}_1=1$ and $\mathit{l}_2$=0). 

Returning to the timing results displayed in Tables \ref{Table:SummarizeR3} and \ref{Table:SummarizeR5}, we recognize that the results may be dominated by a small number of difficult problems.  Even with fixed problem parameters, a problem can be difficult (and require a large amount of iterations to converge) or easy depending on the particular random realization of the test tensor and/or the initial guess. Including the standard deviation helps to describe the effects of this bias; however, the timing results don't account for the problems where the algorithm fails to converge within the prescribed resource limit.  One way to overcome this is to use the performance profiles suggested by Dolan and Mor\'e in \cite{Dolan2002}.  

Suppose that we want to compare the performance of a set of algorithms or solvers $\mathscr{S}$ on a test set $\mathscr{P}$.  Suppose there are $n_s$ algorithms and $n_p$ problems.  For each problem $p \in \mathscr{P}$ and algorithm $s \in \mathscr{S}$ let $t_{p,s}$ be the computing time required to solve problem $p$ using algorithm $s$.  In order to compare algorithms we use the best performance by any algorithm as a baseline and define the performance ratio as 
\beq
r_{p,s} = \frac{t_{p,s}}{\min\{t_{p,s}: \; s \in \mathscr{S}\}}.
\eeq
Although we may be interested in the performance of algorithm $s$ on a given problem $p$, a more insightful analysis can be performed if we can obtain an overall assessment of the algorithm's performance. We can do this by defining the following:
\beq
\rho_s(\tau) = \frac{1}{n_p}\mbox{size}\{p \in \mathscr{P}: r_{p,s} \leq \tau \}.
\eeq
For algorithm $s \in \mathscr{S}$, $\rho_s(\tau)$ is the fraction of problems $p$ for which the performance ratio $r_{p,s}$ is within a factor $\tau \in \mathbb{R}$ of the best ratio (which equals one). Thus, $\rho_s(\tau)$ is the cumulative distribution function for the performance ratio and we refer to it as the performance profile.  By visually examining the performance profiles of each algorithm we can compare the algorithms in $\mathscr{S}$.  In particular, algorithms with large fractions $\rho_s(\tau)$ are preferred.  

Since the performance profile, $\rho_s: \mathbb{R} \mapsto [0,1]$, is a cumulative distribution function it is nondecreasing.  In addition, it is a piecewise constant function, continuous from the right at each breakpoint.  The value of $\rho_s$ at $\tau = 1$ is the fraction of problems for which the algorithm wins over the rest of the algorithms.  In other words,  $\rho_s(1)$ is the fraction of wins for each solver.  For larger values of $\tau$, algorithms with high values of $\rho_s$ relative to the other algorithms indicate robust solvers.  

To examine the performance profiles of each algorithm more easily we group the NCG and PNCG algorithms according to formula for $\beta_{k+1}$, either FR, PR or HS.  Thus 
\begin{eqnarray}
\mathscr{S}_1 &= \{ \mbox{ALS, NCG with $\beta^{FR}$, PNCG with $\widetilde{\beta}^{FR}$, PNCG with $\widehat{\beta}^{FR}$ }\},\\
\mathscr{S}_2 &= \{ \mbox{ALS, NCG with $\beta^{PR}$, PNCG with $\widetilde{\beta}^{PR}$, PNCG with $\widehat{\beta}^{PR}$ }\},\\
\mathscr{S}_3 &= \{ \mbox{ALS, NCG with $\beta^{HS}$, PNCG with $\widetilde{\beta}^{HS}$, PNCG with $\widehat{\beta}^{HS}$ }\}.
\end{eqnarray}

Figure \ref{fig:FR_20} plots the performance profiles of the algorithms in $\mathscr{S}_1$.  In Figure \subref*{fig:FR_20_3_5}, $R = 3$ and $C = 0.5$.  This is an easy problem and from the performance profiles we can see that not only is ALS the fastest it is also the most robust.  We can increase the difficulty of the problem by increasing the collinearity to $0.9$ and Figure \subref*{fig:FR_20_3_9} shows the performance profiles of each algorithm in $\mathscr{S}_1$ when $C = 0.9$.  Since $\rho_s(1)$ indicates what fraction of the 180 trials each algorithm is the fastest, we see that PNCG with $\widetilde{\beta}_{FR}$ is the fastest algorithm in the largest percentage of runs. When $\tau = 3$ approximately $70\%$ of the 180 NCG runs are within three times the fastest time and approximately $40\%$ of the ALS runs are within three times the fastest time.  However, as $\tau$ increases to $10$ we notice that approximately all of the ALS and PNCG runs are within ten times the fastest time but only $90\%$ of the NCG runs are within ten times the fastest time.  This suggests that the NCG algorithm without nonlinear preconditioning is not nearly as robust as the other algorithms.  In Figures \subref*{fig:FR_20_5_5} and \subref*{fig:FR_20_5_9}, $R = 5$ and $C = 0.5$ and 0.9 respectively.  For $C=0.5$, the performance profiles look similar in Figures \subref*{fig:FR_20_3_5} and \subref*{fig:FR_20_5_5} where $R=3$ and  5 respectively.  For $C = 0.9$, the performance profiles in Figure \subref*{fig:FR_20_3_9} where $R=3$ and Figure \subref*{fig:FR_20_5_9} where $R = 5$ differ, however, in both cases, PNCG with $\widetilde{\beta}^{FR}$ is the fastest in the largest percentage of runs and NCG is the least robust algorithm having the smallest value at $\tau = 10$.

Figures \ref{fig:PR_20} and \ref{fig:HS_20} plot the performance profiles for the algorithms in $\mathscr{S}_2$ and $\mathscr{S}_3$, respectively. Once again we see similar results as those displayed in Figure \ref{fig:FR_20}. 

Our next challenge is to examine the performance of the PNCG algorithm when we increase the tensor size.  To better understand the performance, we focus on the results for the algorithms in $\mathscr{S}_2$ since the results are similar for the algorithms in $\mathscr{S}_1$ and $\mathscr{S}_3$.  We consider two different size parameters, $I = 50$ and $I = 100$.  Table \ref{Table:SummarizeI50} reports the timing results when $I = 50$ and Table \ref{Table:SummarizeI100} contains the results when $I = 100$.  As we increase the size of the tensors we see that the results remain similar to the case where $I = 20$.  Regardless of the rank, $R$, the easy problem for which the collinearity is $0.5$, can easily be solved by ALS.  Again, when we move to the more difficult problem of $C =0.9$,  the PNCG algorithms perform the best (except for $I=100$ and $R=5$).  These results are further reflected in the performance profiles shown in Figures \ref{fig:PR_50} and \ref{fig:PR_100}. ALS dominates regardless of rank and size when $C=0.5$, but Figures \subref*{fig:PR_50_3_9}, \subref*{fig:PR_50_5_9}, \subref*{fig:PR_100_3_9}, and \subref*{fig:PR_100_5_9} suggest that PNCG with the $\widetilde{\beta}^{PR}$ variant is the fastest for $C=0.9$ except when $I = 100$ and $R = 5$, where ALS appears faster.  The figures also indicate that the NCG algorithm without nonlinear preconditioning is the least robust.  In the case when $I = 100$, $R = 5$ and the collinearity is $0.9$, we note that Table \ref{Table:SummarizeI100} shows that the ALS algorithm is the fastest on average, while Figure \subref*{fig:PR_100_5_9} shows that the fastest run is most often for PNCG with the $\widetilde{\beta}^{PR}$ variant. Both variants of the PNCG algorithm are more robust than the NCG algorithm, while ALS is the most robust in this case. So we can say that, while PNCG appears significantly faster than ALS for all difficult ($C = 0.9$) problems when the number of factors $R$ and the tensor size $I$ are relatively small, ALS becomes competitive again with PNCG when $R$ and $I$ are large.  Note, however, that the line search parameters in the NCG and PNCG algorithms were the same for every problem, and it may be possible to improve both the NCG and PNCG results by fine-tuning these parameters.  We also see from Tables \ref{Table:SummarizeI50} and \ref{Table:SummarizeI100} that the ability of NCG to successfully recover the original factor matrices is less than both PNCG variants and ALS in some cases. When $I=50$ and $C=0.9$ the difference is small for both $R=3$ and $R=5$.  The difference is more significant when $I = 100$, $C=0.9$ and $R=5$, while there is no difference when $R=3$.  In all cases, the number of successes is essentially the same for both variants of PNCG and ALS.  Thus, the main conclusion from our numerical tests is that nonlinear preconditioning can dramatically improve the speed and robustness of NCG: PNCG is significantly faster and more robust than NCG for all difficult ($C = 0.9$) CP problems we tested.

\section{Conclusion}

We have proposed an algorithm for computing the canonical rank-$R$ tensor decomposition that applies ALS as a nonlinear preconditioner to the nonlinear conjugate gradient algorithm.  We consider the ALS algorithm as a preconditioner because it is the standard algorithm used to compute the canonical rank-$R$ tensor decomposition but it is known to converge very slowly for certain problems, for which acceleration by NCG is expected to be beneficial.  We have considered several approaches for incorporating the nonlinear preconditioner into the NCG algorithm that have been described in the literature \cite{BartelsDaniel1974,CGO1977,Mittlemann1980,Elad2010,PETSC2013}, corresponding to two different sets of preconditioned formulas for the standard FR, PR and HS update parameter, $\beta$, namely the $\widetilde{\beta}$ and $\widehat{\beta}$ formulas.  If we use the $\widehat{\beta}$ formulas and apply the PNCG algorithm using an SPD preconditioner to a convex quadratic function using an exact line search, then the PNCG algorithm simplifies to the PCG  algorithm.  Also, we proved a new convergence result for one of the PNCG variants under suitable conditions, building on known convergence results for non-preconditioned NCG when line searches are used that satisfy the strong Wolfe conditions.
Note that it is very easy to extend existing NCG software with the nonlinear preconditioning mechanism. Our simulation code and examples can be found at \url{www.math.uwaterloo.ca/~hdesterc/pncg.html}.

Following the methodology of \cite{TomasiBro2006} we create numerous test tensors and perform extensive numerical tests comparing the PNCG algorithm to the ALS and NCG algorithms.  We consider a wide range of tensor sizes, ranks, factor collinearity and noise levels.  Results in \cite{ADK2011} showed that ALS is normally faster than NCG. In this paper, we show that NCG preconditioned with ALS (or, equivalently, ALS accelerated by NCG) is often significantly faster than ALS by itself, for difficult problems. For easy problems, where the collinearity is 0.5, ALS outperforms all other algorithms. However, when the problem becomes more difficult and the collinearity is 0.9, the PNCG algorithm is often the fastest algorithm.  The only case where ALS is faster is when we consider our largest tensor size and highest rank.  The performance profiles of each algorithm also show that for the more difficult problems, PNCG is consistently both more robust and faster than the NCG algorithm. For our optimization problems, we generally obtain convergent results for all of the six variants of the PNCG algorithm we considered. It is interesting that for the PDE problems of \cite{PETSC2013}, out of the $\widetilde{\beta}$ variants, only $\widetilde{\beta}^{PR}$ was found viable. It appears that the $\widehat{\beta}$ variants were not investigated in \cite{PETSC2013}. We did find for our test tensors that the $\widetilde{\beta}^{PR}$ formula, which does not reduce to PCG in the linear case, converges the fastest for most cases.

The PNCG algorithm discussed in this paper is formulated under a general framework. While this approach has met with success previously in certain application areas \cite{BartelsDaniel1974,CGO1977,Mittlemann1980,Elad2010,PETSC2013} and may offer promising avenues for further applications, it appears that the nonlinearly preconditioned NCG approach has received relatively little attention in the broader community and remains underexplored both theoretically and experimentally. It will be interesting to investigate the effectiveness of PNCG for other nonlinear optimization problems.  Other nonlinear least-squares optimization problems for which ALS solvers are available are good initial candidates for further study.  However, as with PCG for SPD linear systems \cite{Saad2003}, it is fully expected that devising effective preconditioners for more general nonlinear optimization problems will be highly problem-dependent while at the same time being crucial for gaining substantial performance benefits.  

\nocite{*}
\bibliographystyle{plain}
\bibliography{PNCGPaper}	

\newpage

{\renewcommand{\arraystretch}{3}
\begin{table}[p]
\centering
\caption{Variants of $\overline{\beta}_{k+1}$ for the Nonlinearly Preconditioned Nonlinear Conjugate Gradient Algorithm (PNCG).}
\label{Table:BetaFormulas}
\begin{tabular}{|c|c|c|}
\hline
 & $\widetilde{\beta}_{k+1}$ & $\widehat{\beta}_{k+1}$ \\
\hline\hline
	 Fletcher-Reeves & \raise.9ex\hbox{$\displaystyle \widetilde{\beta}_{k+1}^{FR} = \frac{\overline{\bfg}_{k+1}^T\overline{\bfg}_{k+1}}{\overline{\bfg}_{k}^T\overline{\bfg}_{k}}$} &  \raise.9ex\hbox{$\displaystyle \widehat{\beta}_{k+1}^{FR} = \frac{{\bfg}_{k+1}^T\overline{\bfg}_{k+1}}{{\bfg}_{k}^T\overline{\bfg}_{k}}$} \\
\hline
 Polak-Ribi\`{e}re &  \raise.9ex\hbox{$\displaystyle \widetilde{\beta}_{k+1}^{PR} = \frac{\overline{\bfg}_{k+1}^T(\overline{\bfg}_{k+1}-\overline{\bfg}_k)}{\overline{\bfg}_{k}^T\overline{\bfg}_{k}}$} &  \raise.9ex\hbox{$\displaystyle \widehat{\beta}_{k+1}^{PR} = \frac{{\bfg}_{k+1}^T(\overline{\bfg}_{k+1}-\overline{\bfg}_k)}{{\bfg}_{k}^T\overline{\bfg}_{k}}$}  \\
\hline
Hestenes-Stiefel &  \raise.9ex\hbox{$\displaystyle \widetilde{\beta}_{k+1}^{HS} = \frac{\overline{\bfg}_{k+1}^T(\overline{\bfg}_{k+1}-\overline{\bfg}_k)}{(\overline{\bfg}_{k+1}-\overline{\bfg}_k)^T\bfp_k}$} &  \raise.9ex\hbox{$\displaystyle \widehat{\beta}_{k+1}^{HS} = \frac{{\bfg}_{k+1}^T(\overline{\bfg}_{k+1}-\overline{\bfg}_k)}{({\bfg}_{k+1}-{\bfg}_k)^T\bfp_k}$} \\
\hline
\end{tabular}
\end{table}}

\clearpage
{\renewcommand{\arraystretch}{1.45}
\input{Table_20_3_Paper}}

\clearpage
{\renewcommand{\arraystretch}{1.45}
\input{Table_20_5_Paper}}

\clearpage
{\renewcommand{\arraystretch}{1.45}
\input{Table_50_Paper}}

\clearpage
{\renewcommand{\arraystretch}{1.45}
\input{Table_100_Paper}

\clearpage

\begin{figure}[p]
\captionsetup{justification=centering}
\centering
\subfloat[$R = 3$, Collinearity$ = 0.5$]{\label{fig:FR_20_3_5}
\includegraphics[width=0.4\linewidth]{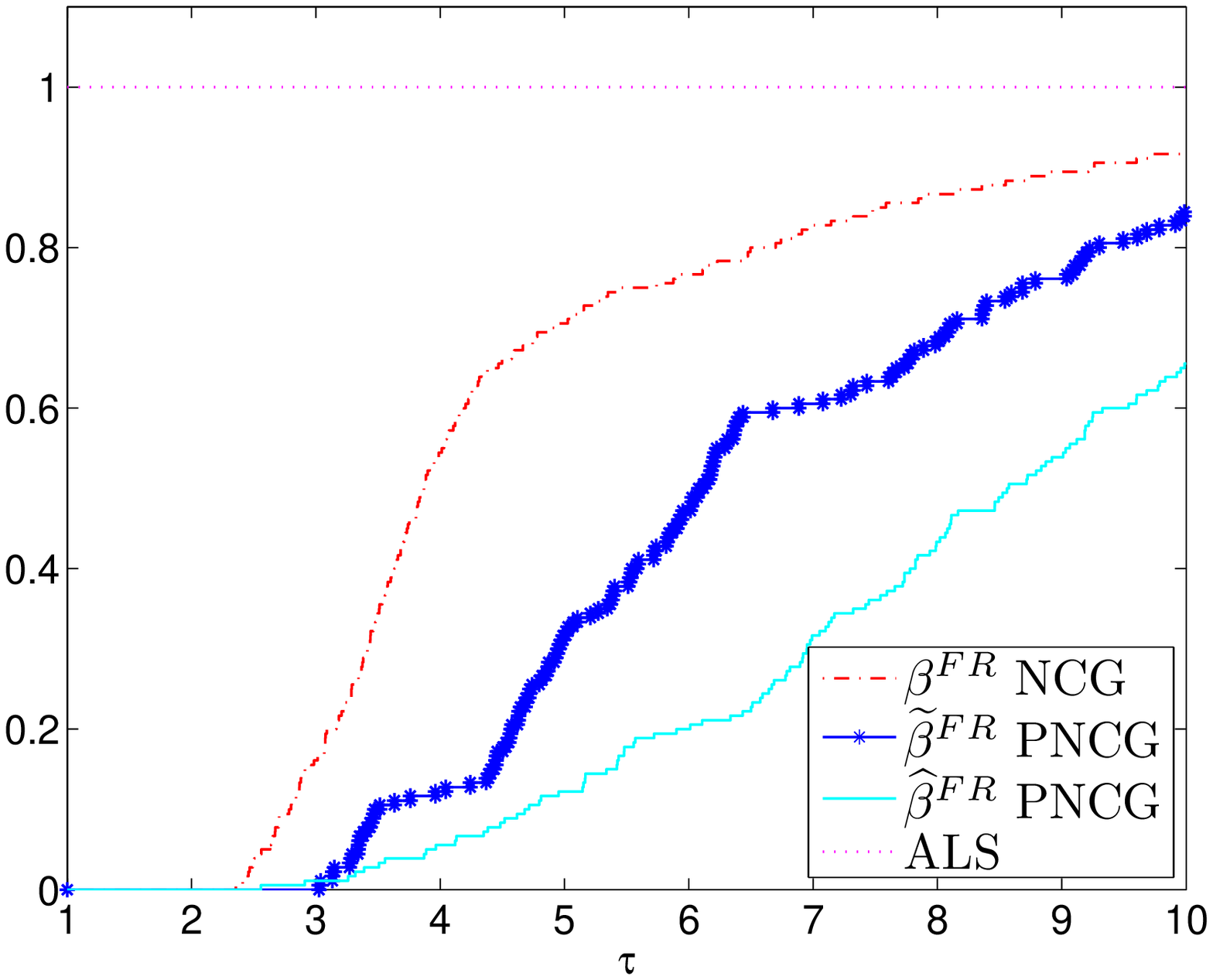}}
\hspace{0.1\linewidth}
\subfloat[$R = 3$, Collinearity$ = 0.9$]{\label{fig:FR_20_3_9}
\includegraphics[width=0.4\linewidth]{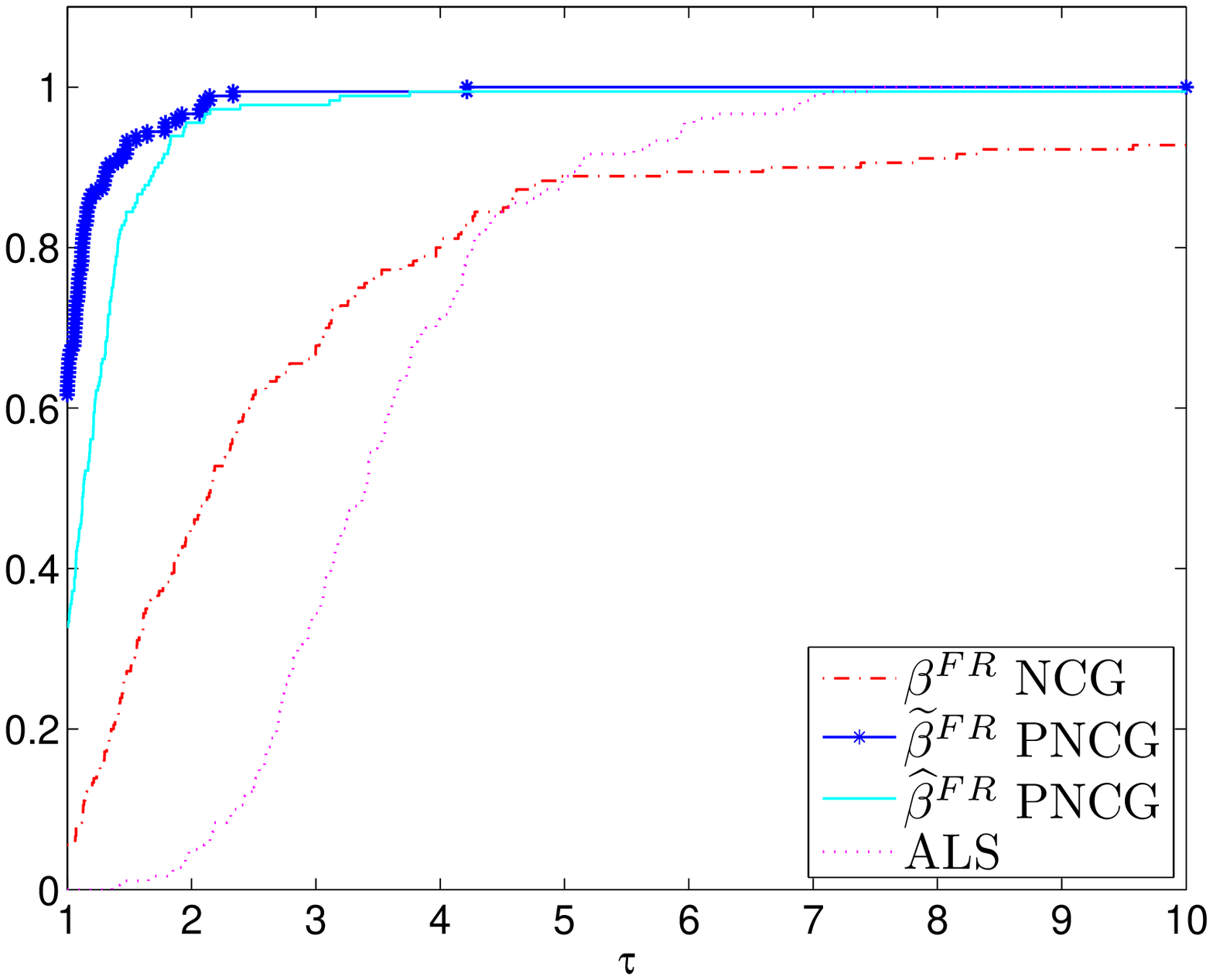}}
\\[20pt]
\subfloat[$R = 5$, Collinearity$ = 0.5$]{\label{fig:FR_20_5_5}
\includegraphics[width=0.4\linewidth]{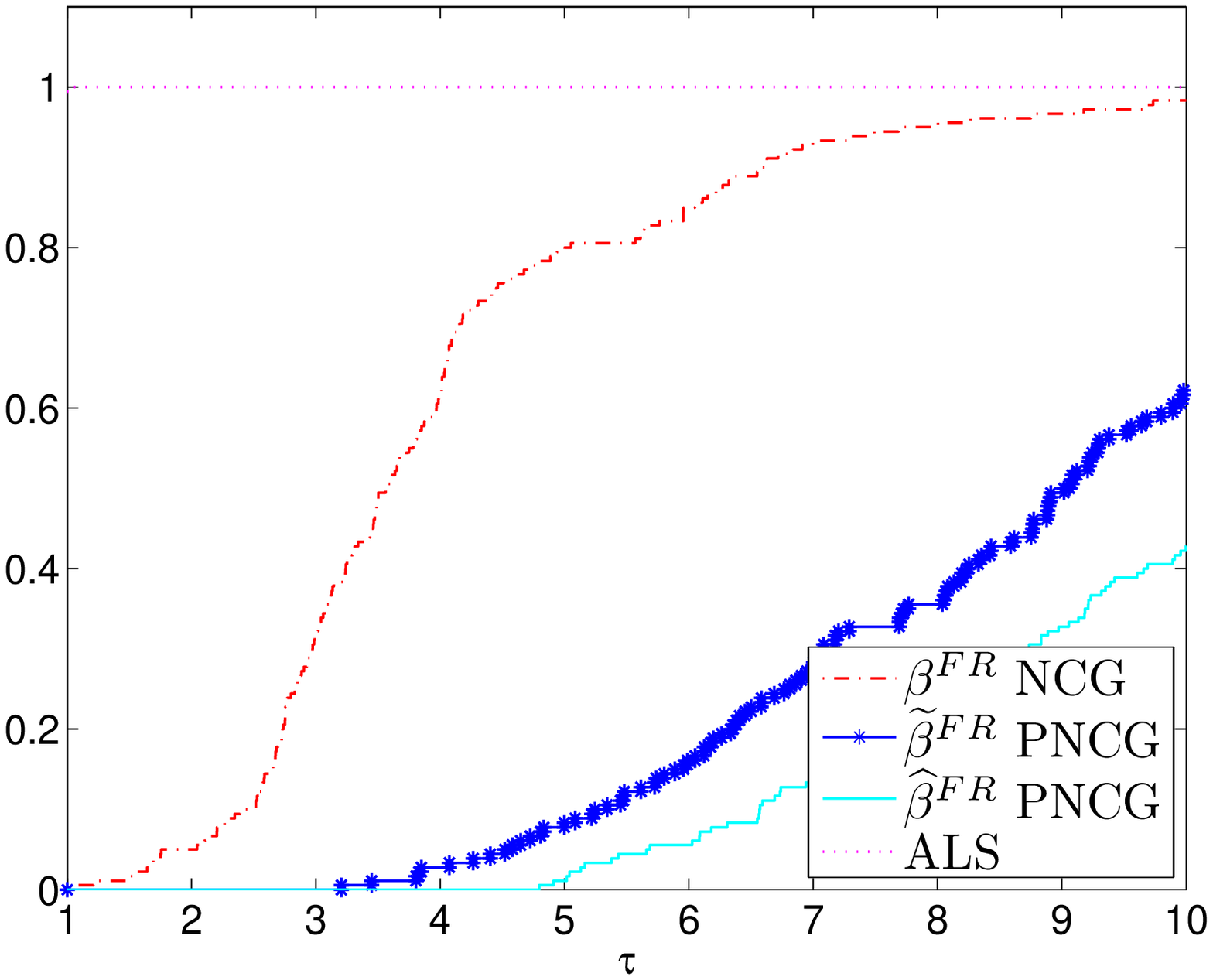}}
\hspace{0.1\linewidth}
\subfloat[$R = 5$, Collinearity$ = 0.9$]{\label{fig:FR_20_5_9}
\includegraphics[width=0.4\linewidth]{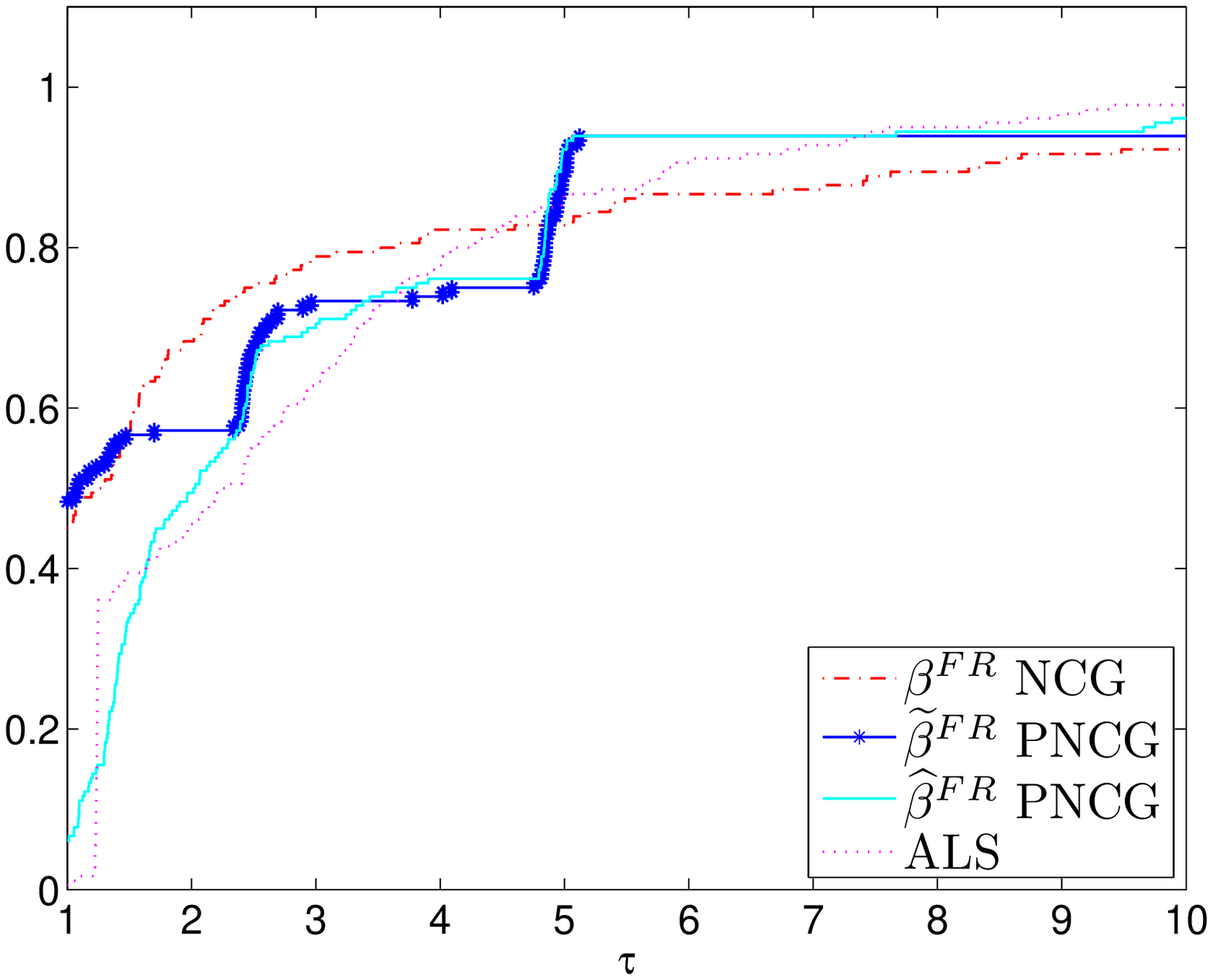}}
\caption{Performance profiles for the algorithms in $\mathscr{S}_1$ with $I = 20$.}
\label{fig:FR_20}
\end{figure}

\begin{figure}[p]
\captionsetup{justification=centering}
\centering
\subfloat[$R = 3$, Collinearity$ = 0.5$]{\label{fig:PR_20_3_5}
\includegraphics[width=0.4\linewidth]{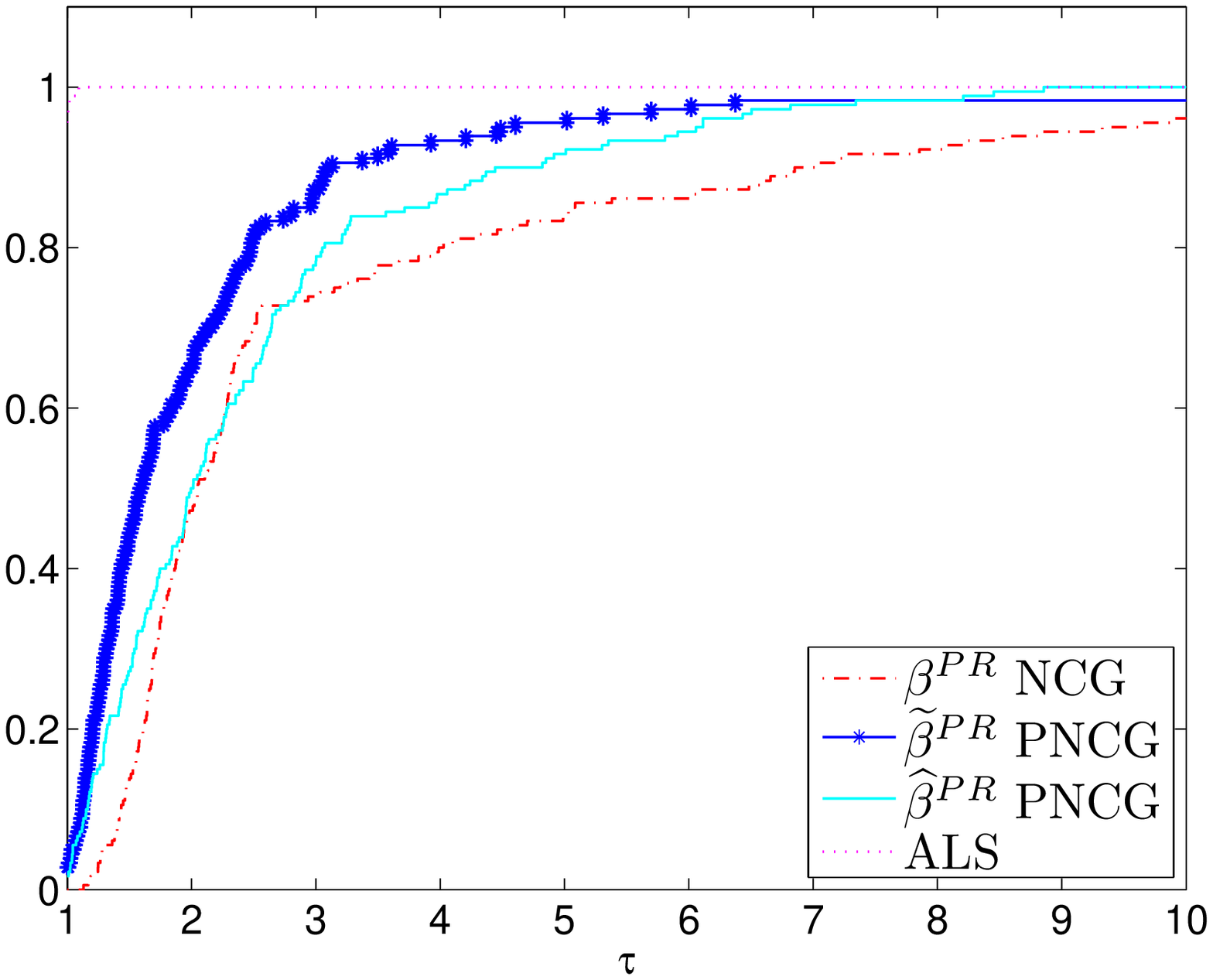}}
\hspace{0.1\linewidth}
\subfloat[$R = 3$, Collinearity$ = 0.9$]{\label{fig:PR_20_3_9}
\includegraphics[width=0.4\linewidth]{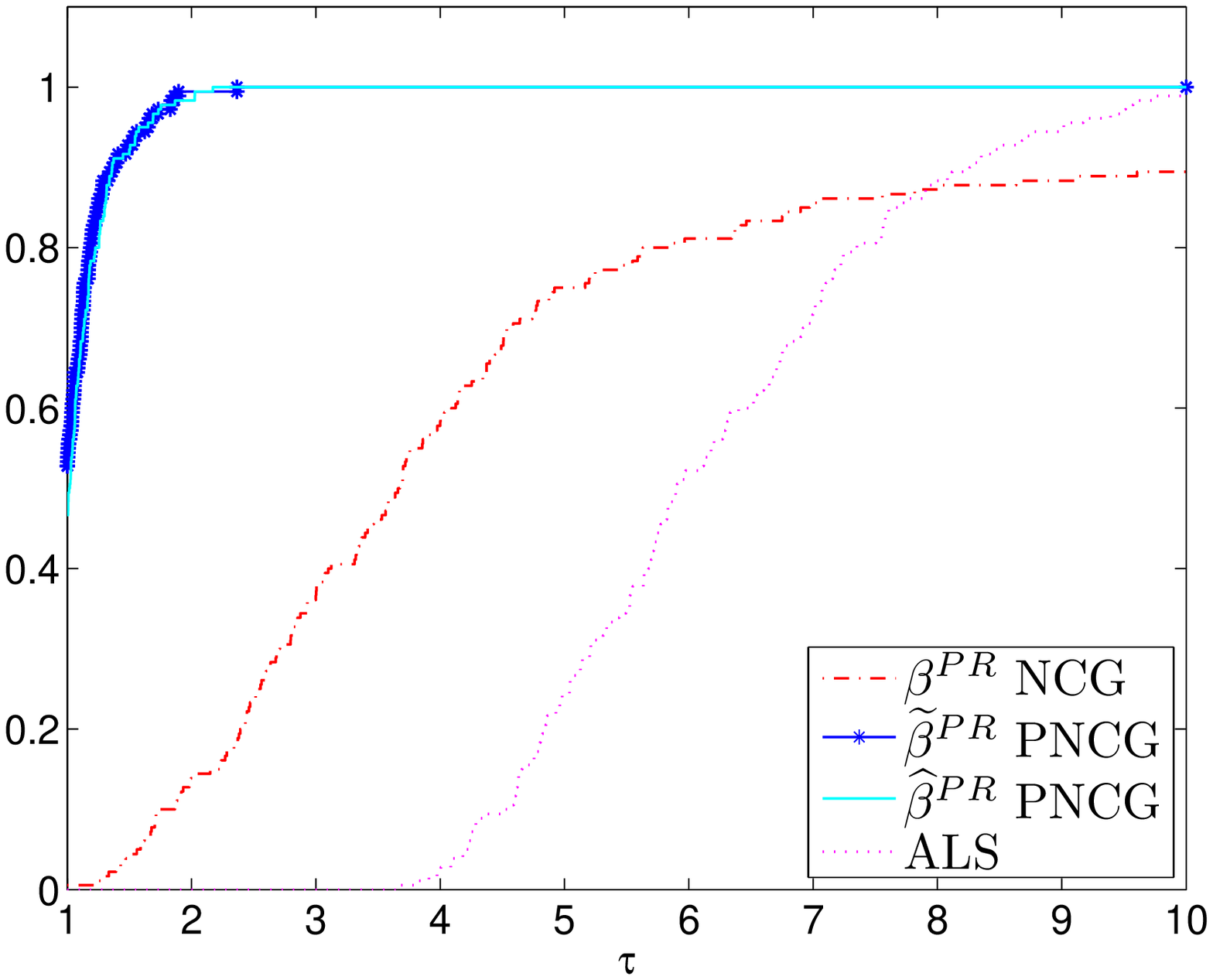}}
\\[20pt]
\subfloat[$R = 5$, Collinearity$ = 0.5$]{\label{fig:PR_20_5_5}
\includegraphics[width=0.4\linewidth]{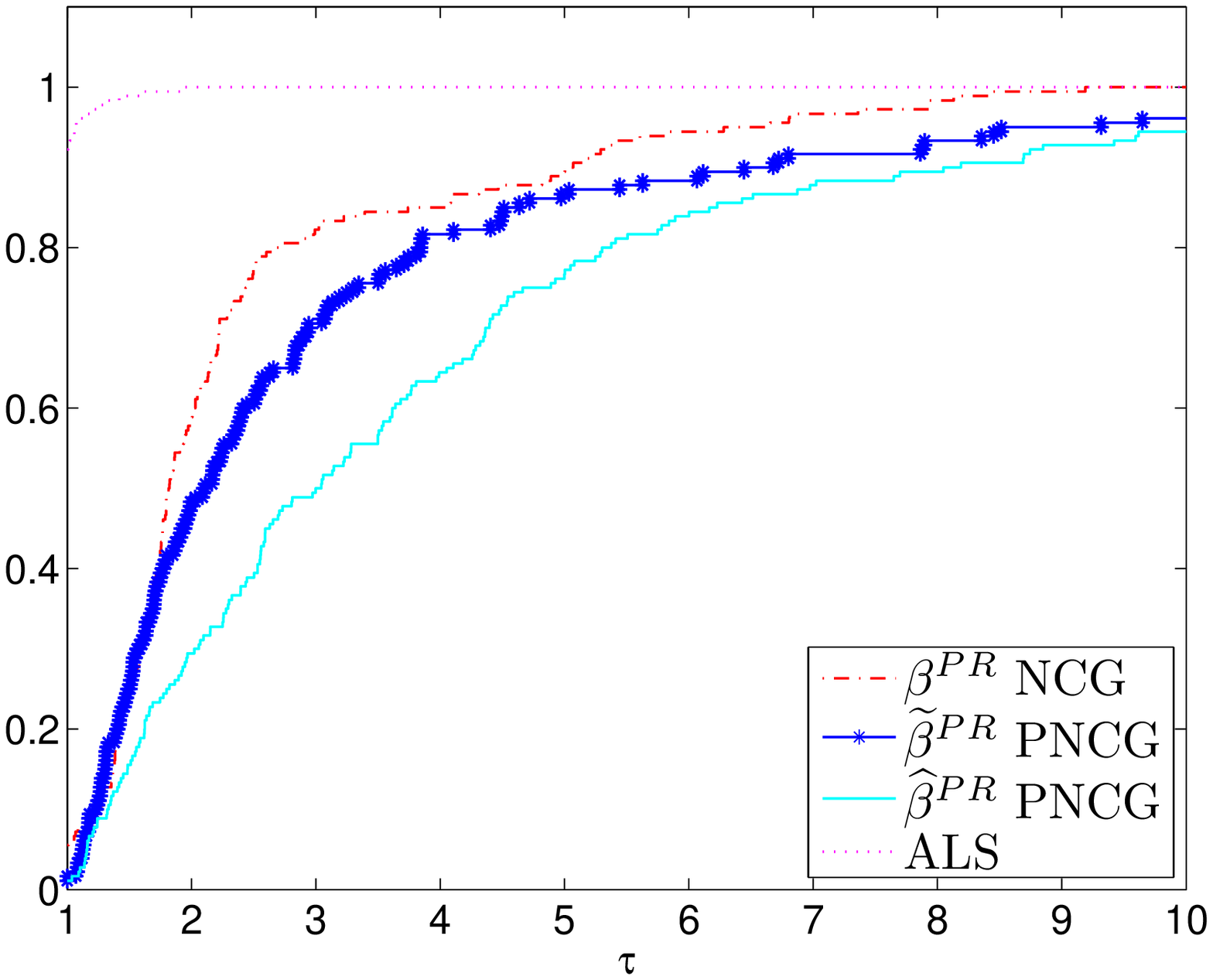}}
\hspace{0.1\linewidth}
\subfloat[$R = 5$, Collinearity$ = 0.9$]{\label{fig:PR_20_5_9}
\includegraphics[width=0.4\linewidth]{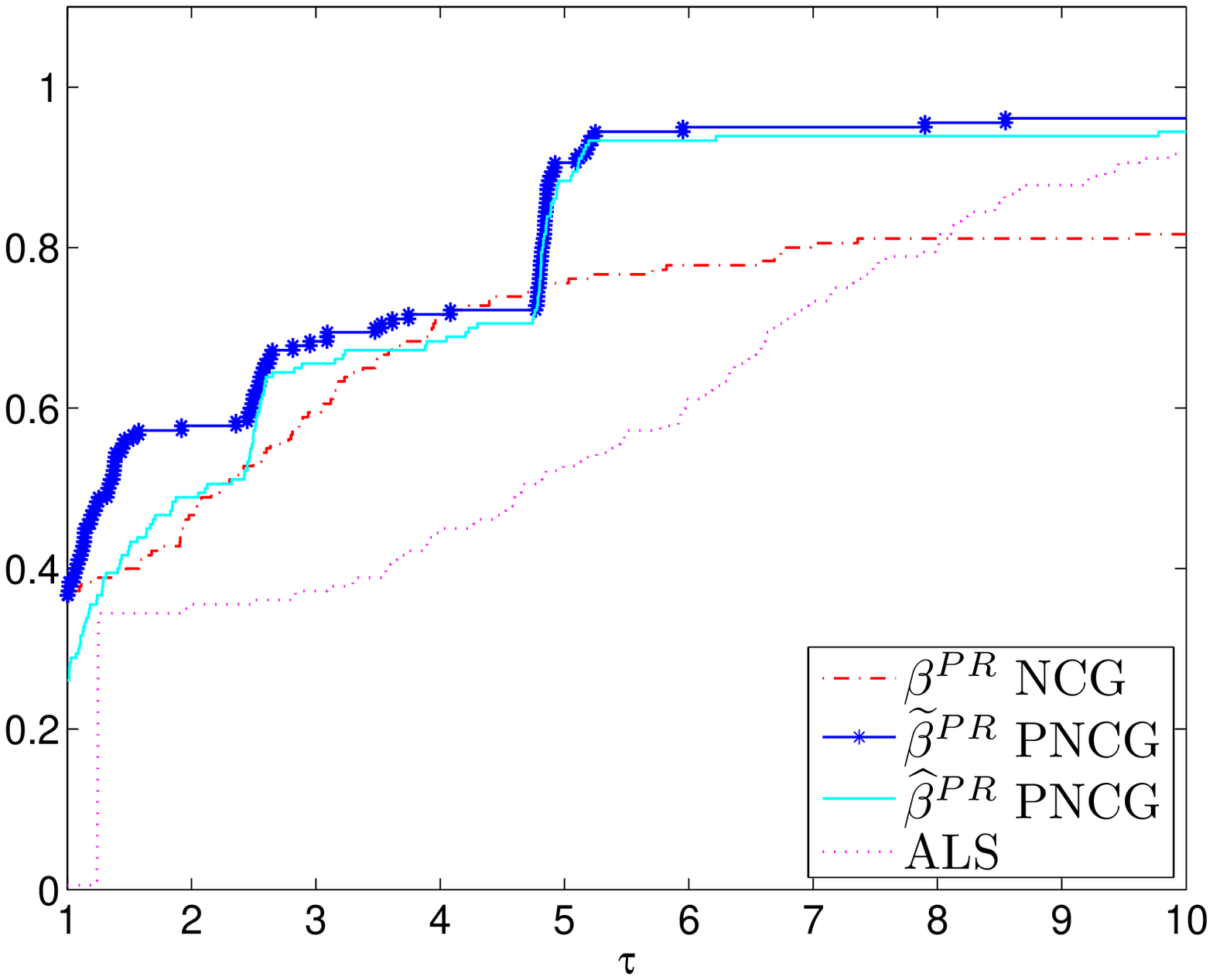}}
\caption{Performance profiles for the algorithms in $\mathscr{S}_2$ with $I = 20$.}
\label{fig:PR_20}
\end{figure}

\begin{figure}[p]
\captionsetup{justification=centering}
\centering
\subfloat[$R = 3$, Collinearity$ = 0.5$]{\label{fig:HS_20_3_5}
\includegraphics[width=0.4\linewidth]{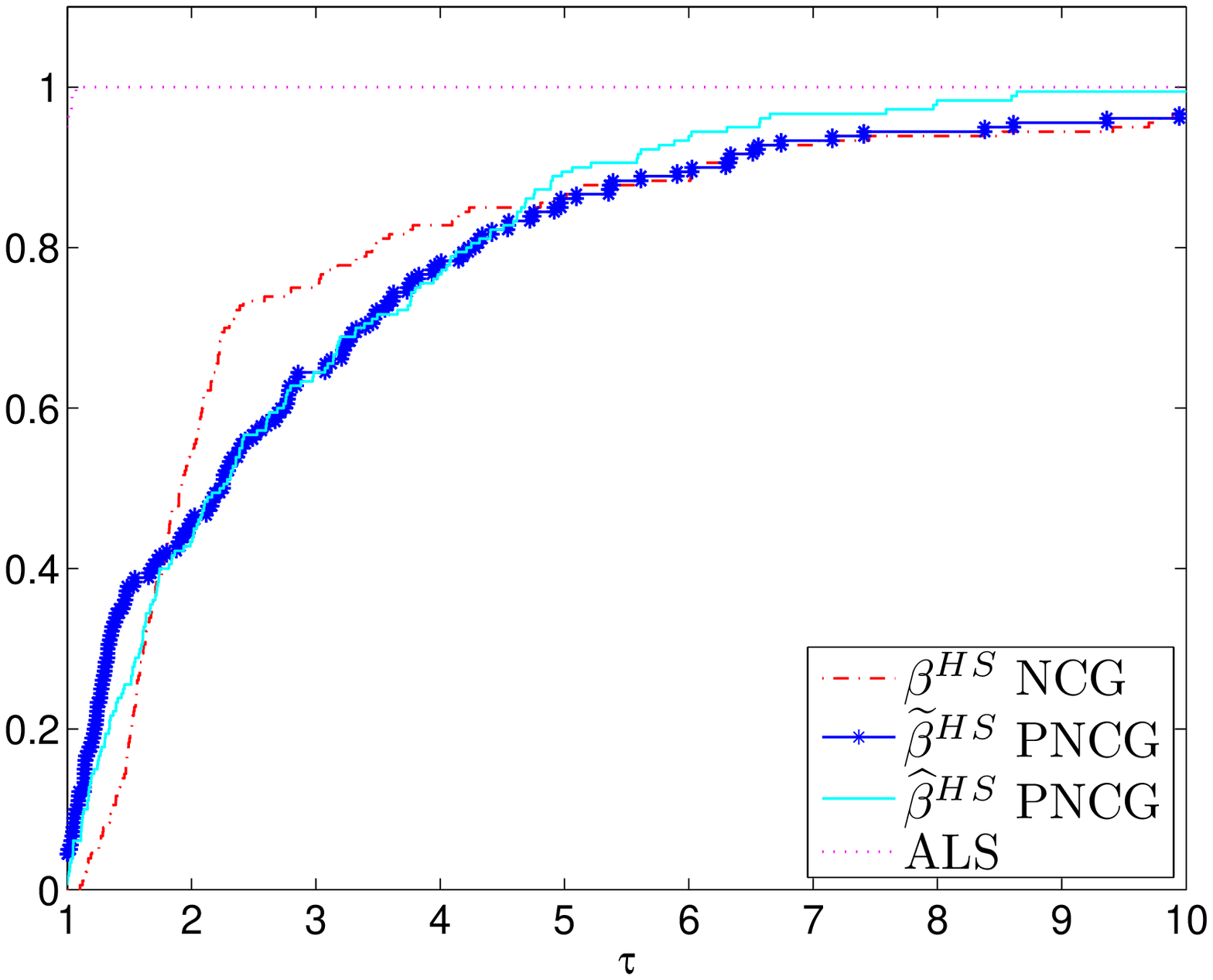}}
\hspace{0.1\linewidth}
\subfloat[$R = 3$, Collinearity$ = 0.9$]{\label{fig:HS_20_3_9}
\includegraphics[width=0.4\linewidth]{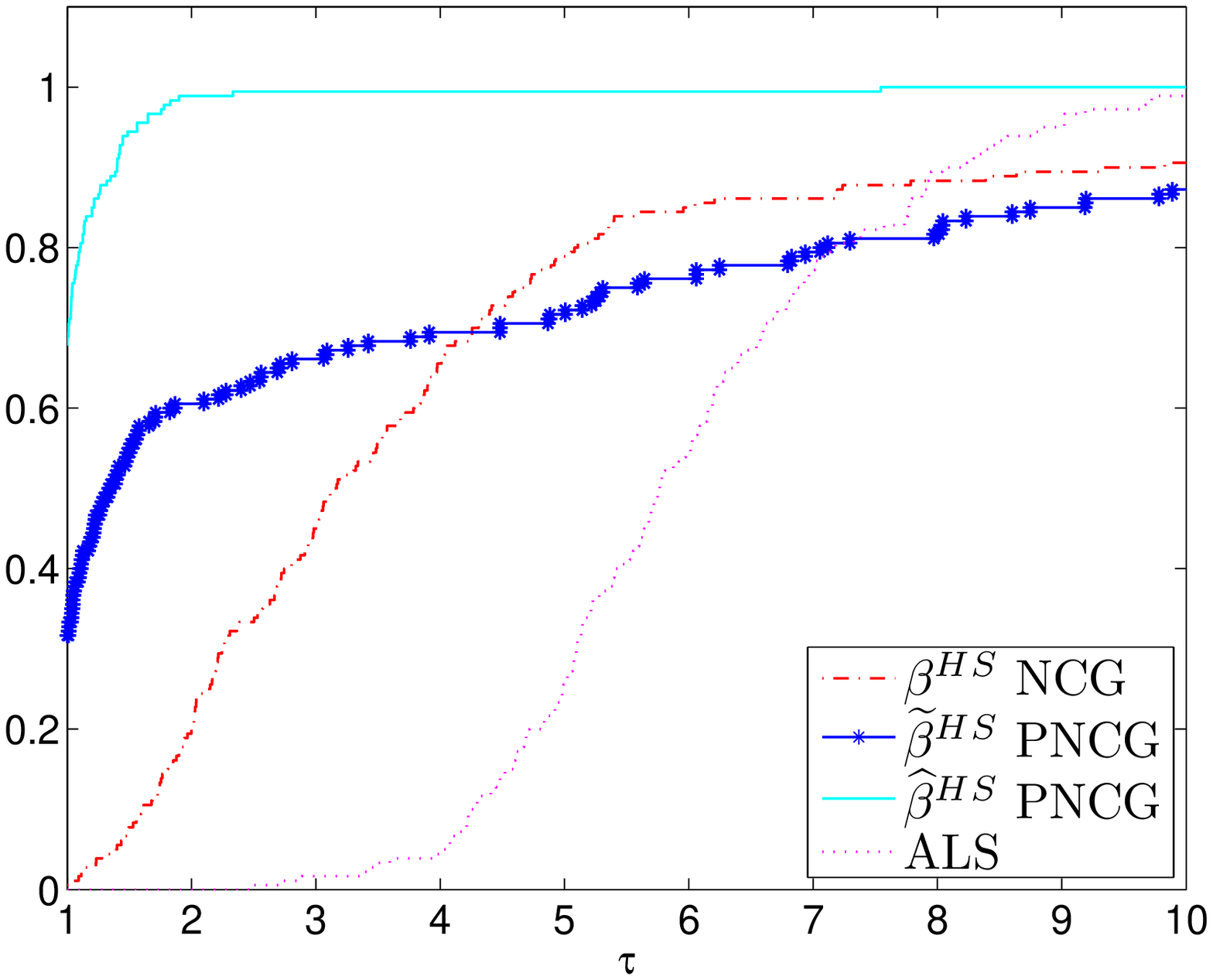}}
\\[20pt]
\subfloat[$R = 5$, Collinearity$ = 0.5$]{\label{fig:HS_20_5_5}
\includegraphics[width=0.4\linewidth]{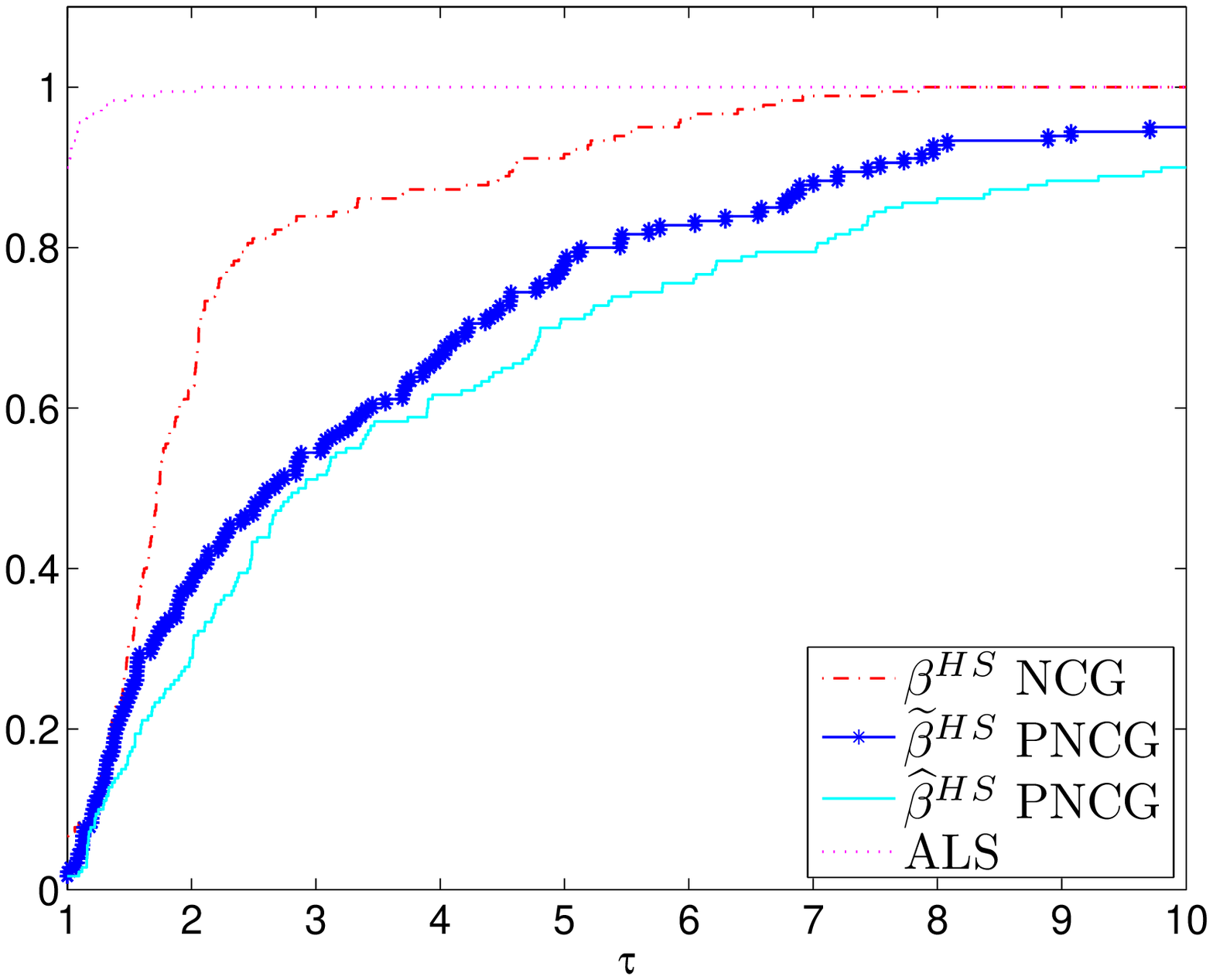}}
\hspace{0.1\linewidth}
\subfloat[$R = 5$, Collinearity$ = 0.9$]{\label{fig:HS_20_5_9}
\includegraphics[width=0.4\linewidth]{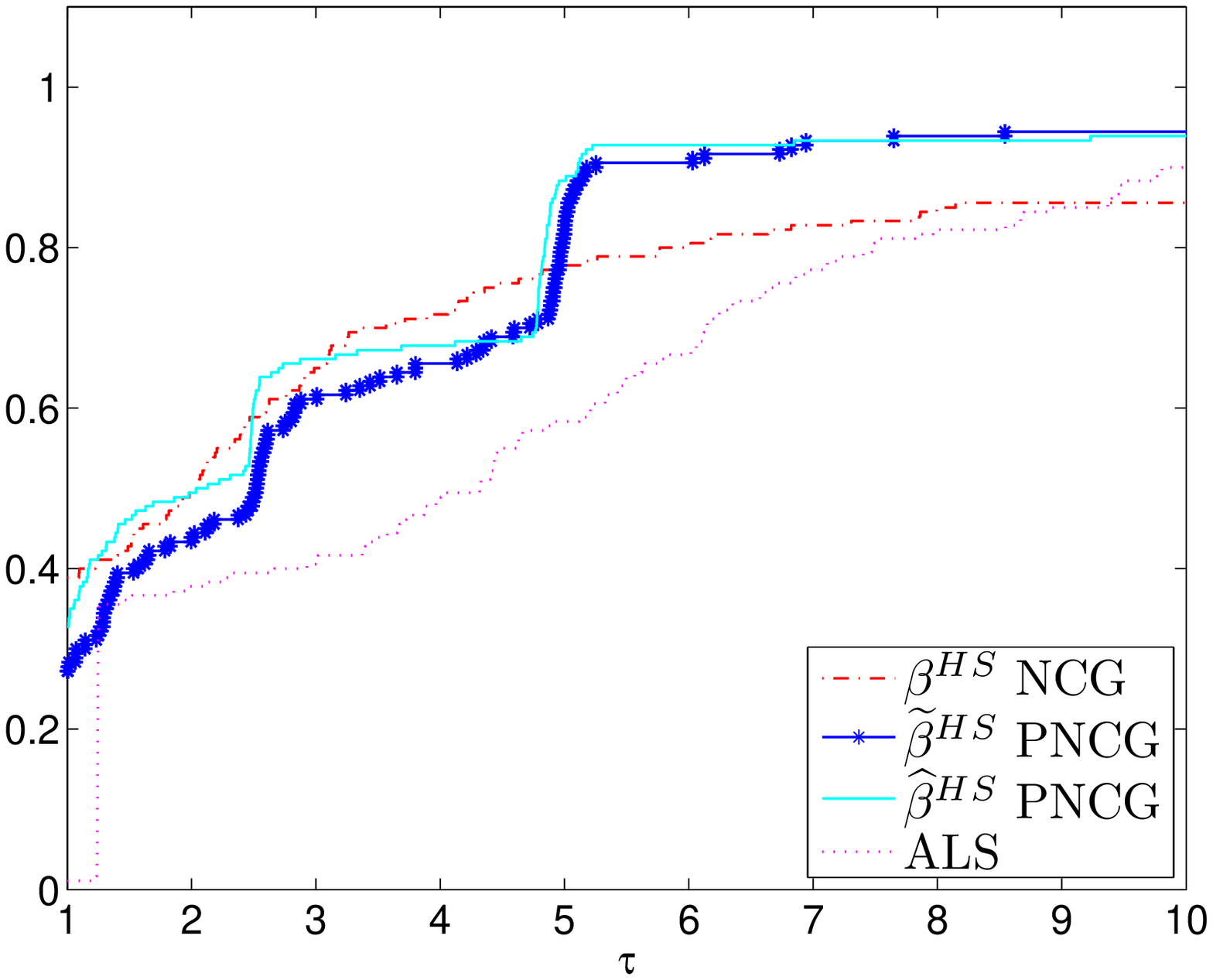}}
\caption{Performance profiles for the algorithms in $\mathscr{S}_3$ with $I = 20$.}
\label{fig:HS_20}
\end{figure}

\begin{figure}[p]
\captionsetup{justification=centering}
\centering
\subfloat[$R = 3$, Collinearity$ = 0.5$]{\label{fig:PR_50_3_5}
\includegraphics[width=0.4\linewidth]{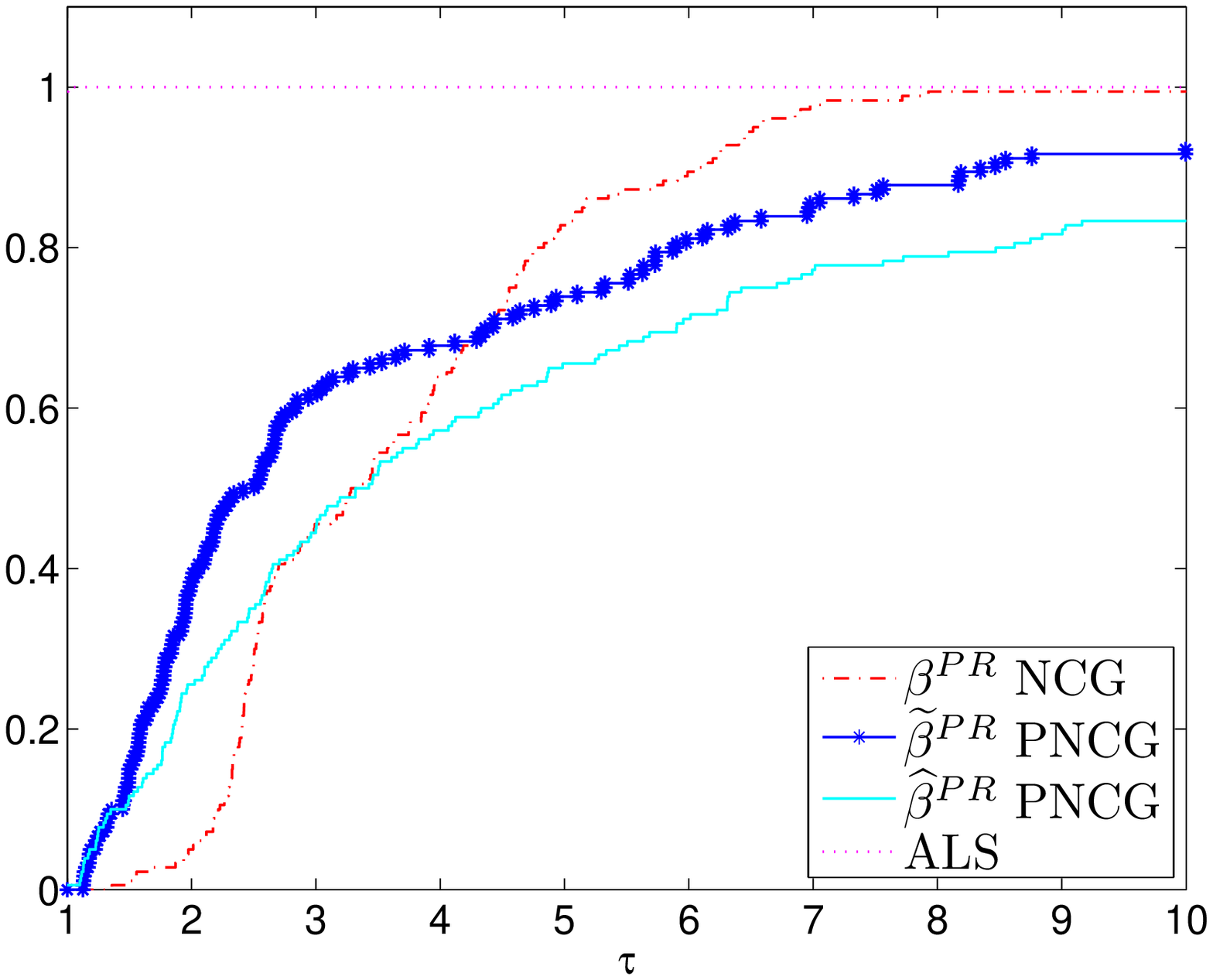}}
\hspace{0.1\linewidth}
\subfloat[$R = 3$, Collinearity$ = 0.9$]{\label{fig:PR_50_3_9}
\includegraphics[width=0.4\linewidth]{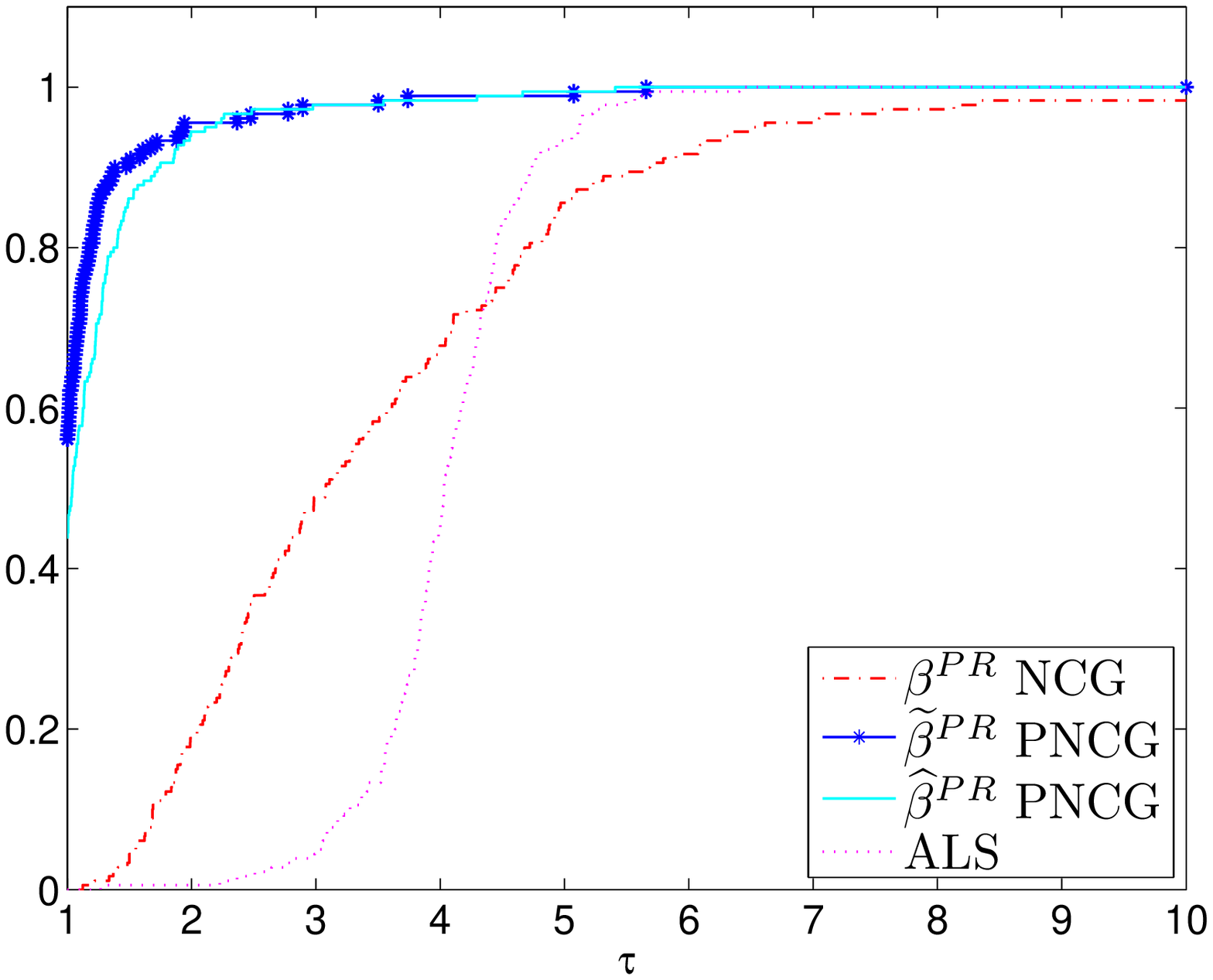}}
\\[20pt]
\subfloat[$R = 5$, Collinearity$ = 0.5$]{\label{fig:PR_50_5_5}
\includegraphics[width=0.4\linewidth]{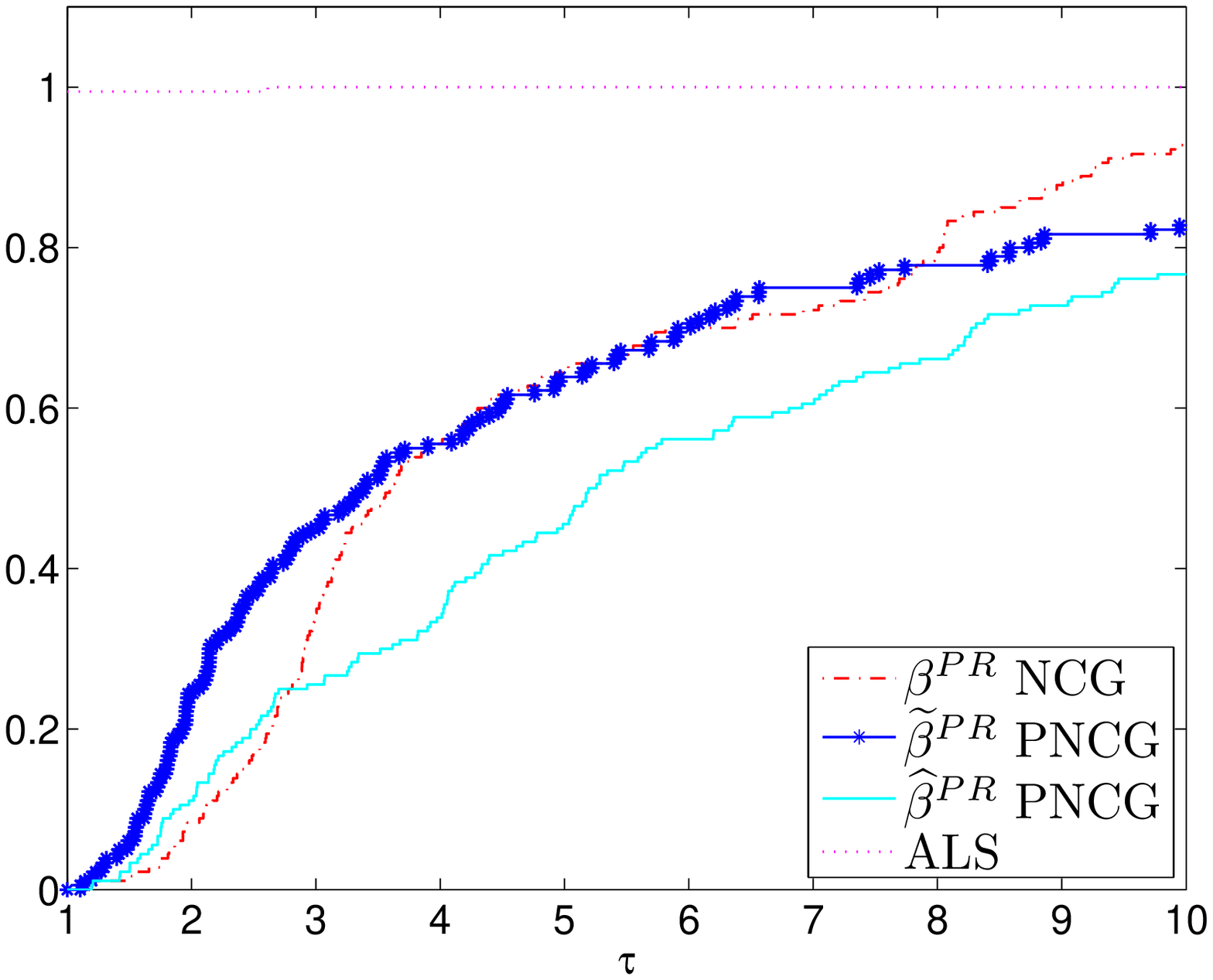}}
\hspace{0.1\linewidth}
\subfloat[$R = 5$, Collinearity$ = 0.9$]{\label{fig:PR_50_5_9}
\includegraphics[width=0.4\linewidth]{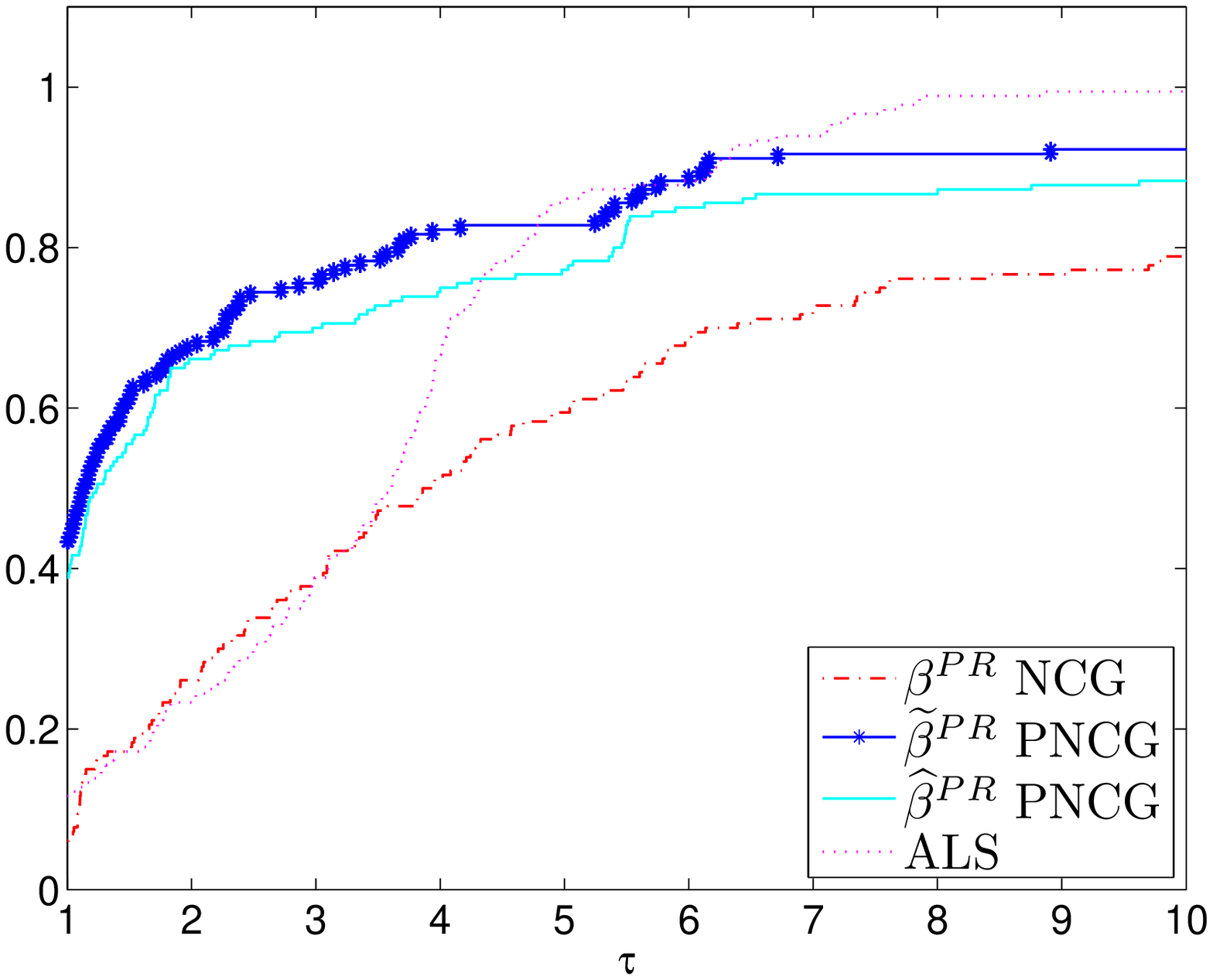}}
\caption{Performance profiles for the algorithms in $\mathscr{S}_2$ with $I = 50$.}
\label{fig:PR_50}
\end{figure}

\begin{figure}[p]
\captionsetup{justification=centering}
\centering
\subfloat[$R = 3$, Collinearity$ = 0.5$]{\label{fig:PR_100_3_5}
\includegraphics[width=0.4\linewidth]{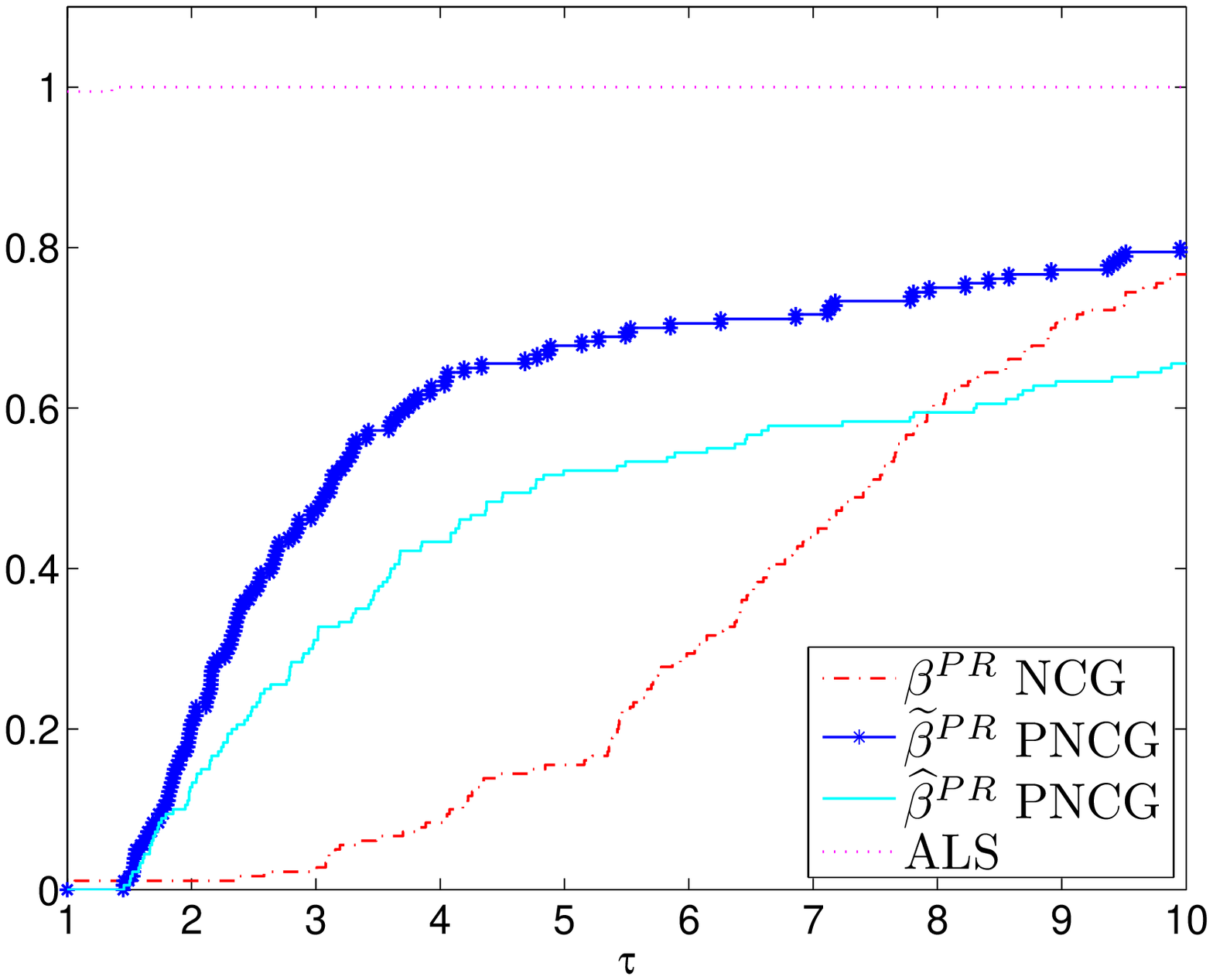}}
\hspace{0.1\linewidth}
\subfloat[$R = 3$, Collinearity$ = 0.9$]{\label{fig:PR_100_3_9}
\includegraphics[width=0.4\linewidth]{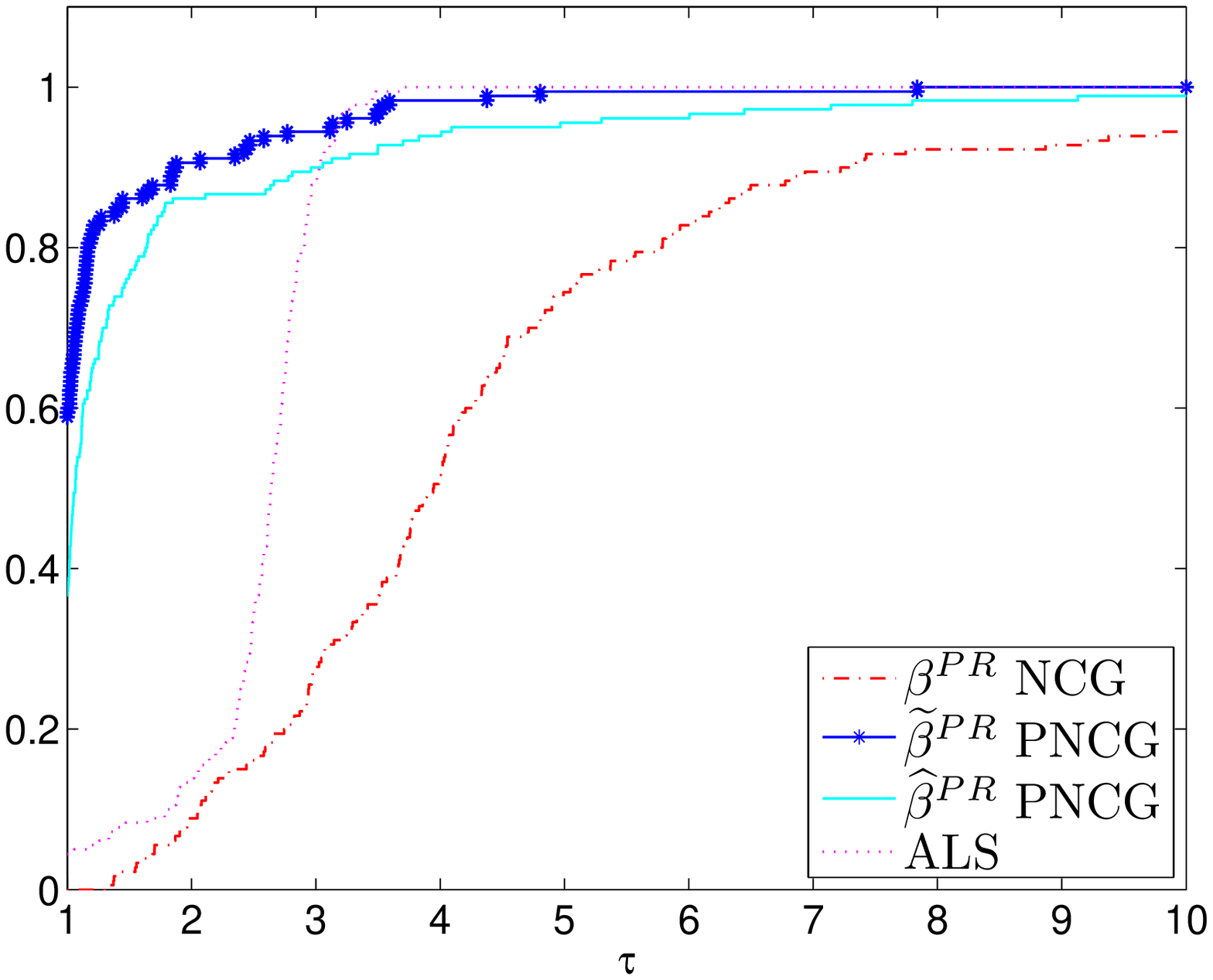}}
\\[20pt]
\subfloat[$R = 5$, Collinearity$ = 0.5$]{\label{fig:PR_100_5_5}
\includegraphics[width=0.4\linewidth]{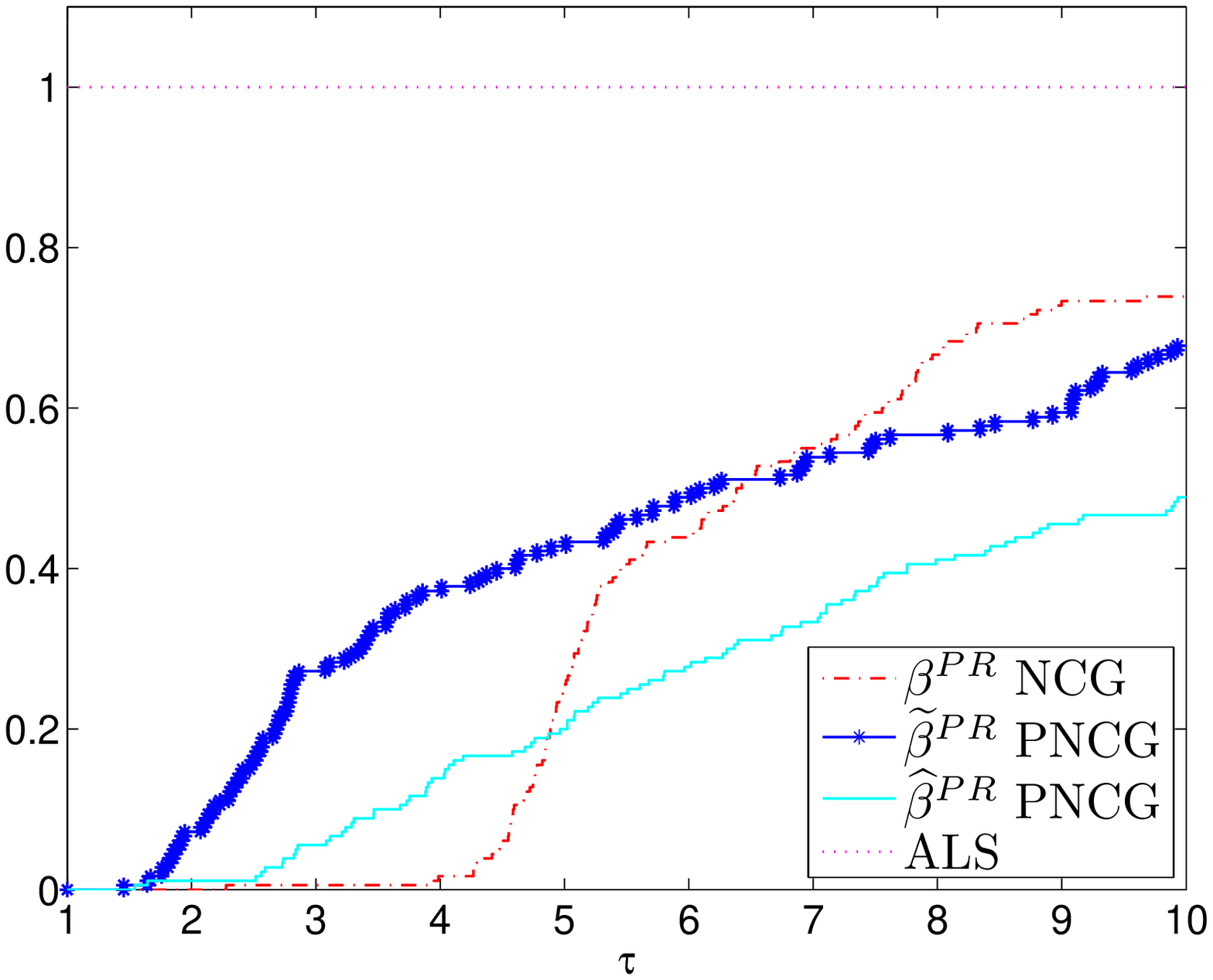}}
\hspace{0.1\linewidth}
\subfloat[$R = 5$, Collinearity$ = 0.9$]{\label{fig:PR_100_5_9}
\includegraphics[width=0.4\linewidth]{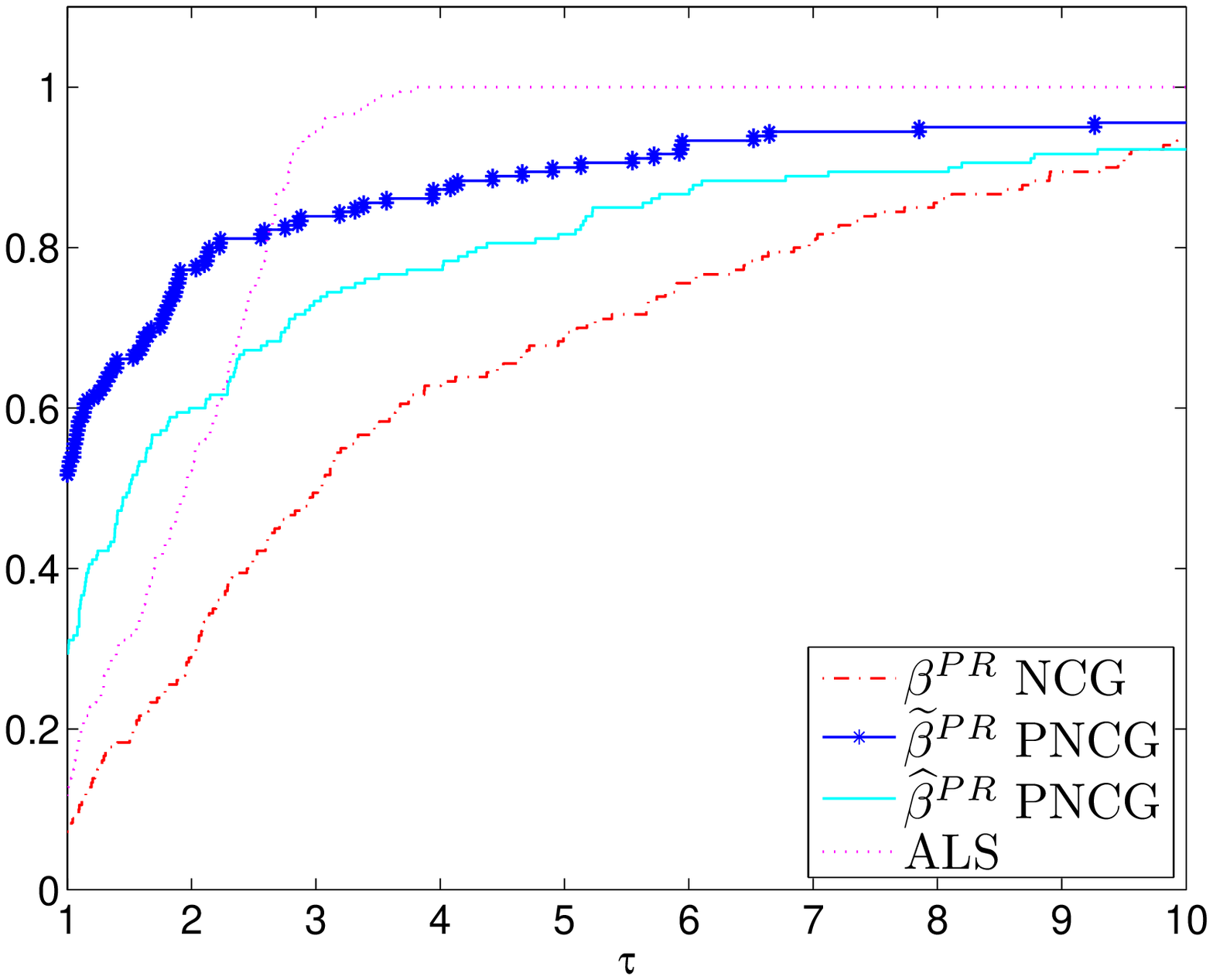}}
\caption{Performance profiles for the algorithms in $\mathscr{S}_2$ with $I = 100$.}
\label{fig:PR_100}
\end{figure}

\end{document}

%% file: Table_20_3_Paper.tex
\begin{table}[p]  
\centering
\caption{Speed Comparison with $R = 3$ and $I=20$.}
\label{Table:SummarizeR3}
\begin{tabular}{| l | c | c |}
\hline
\multirow{2}{*}{Optimization Method} & \multicolumn{2}{|c|}{Mean Time (Seconds)} \\
\cline{2-3}
& $C = 0.5$ & $C = 0.9$\\
\hline
\hline
ALS& \textbf{0.1644} $\pm$  \textbf{0.0185} \textbf{(180)} \textbf{(180)}& 5.3182 $\pm$  1.1356 \textbf{(180)} \textbf{(99)}\\
\hline
NCG - $\beta^{FR}$& 0.8617 $\pm$  0.6658 \textbf{(178)} \textbf{(180)}& 4.1649 $\pm$  3.8092 \textbf{(172)} \textbf{(98)}\\
\hline
PNCG - $\widetilde{\beta}^{FR}$& 1.1707 $\pm$  0.5962 \textbf{(180)} \textbf{(180)}& 1.7556 $\pm$  0.5792 \textbf{(180)} \textbf{(99)}\\
\hline
PNCG - $\widehat{\beta}^{FR}$& 1.5308 $\pm$  0.7196 \textbf{(180)} \textbf{(180)}& 2.1131 $\pm$  1.3339 \textbf{(180)} \textbf{(99)}\\
\hline
NCG - $\beta^{PR}$& 0.5170 $\pm$  0.5300 \textbf{(179)} \textbf{(180)}& 3.5328 $\pm$  2.7377 \textbf{(167)} \textbf{(97)}\\
\hline
PNCG - $\widetilde{\beta}^{PR}$& 0.3434 $\pm$  0.2611 \textbf{(180)} \textbf{(180)}& \textbf{0.9676} $\pm$  \textbf{0.2020} \textbf{(180)} \textbf{(99)}\\
\hline
PNCG - $\widehat{\beta}^{PR}$& 0.4087 $\pm$  0.2592 \textbf{(180)} \textbf{(180)}& 0.9979 $\pm$  0.3077 \textbf{(180)} \textbf{(99)}\\
\hline
NCG - $\beta^{HS}$& 0.4457 $\pm$  0.3458 \textbf{(178)} \textbf{(180)}& 3.1265 $\pm$  2.0725 \textbf{(167)} \textbf{(98)}\\
\hline
PNCG - $\widetilde{\beta}^{HS}$& 0.4969 $\pm$  0.4234 \textbf{(180)} \textbf{(180)}& 3.3269 $\pm$  3.6214 \textbf{(180)} \textbf{(99)}\\
\hline
PNCG - $\widehat{\beta}^{HS}$& 0.4675 $\pm$  0.3095 \textbf{(180)} \textbf{(180)}& 1.0267 $\pm$  0.4513 \textbf{(180)} \textbf{(99)}\\
\hline
\end{tabular}
\end{table}

%% file: Table_20_5_Paper.tex
\begin{table}[p]  
\centering
\caption{Speed Comparison with $R = 5$ and $I=20$.}
\label{Table:SummarizeR5}
\begin{tabular}{| l | c | c |}
\hline
\multirow{2}{*}{Optimization Method} & \multicolumn{2}{|c|}{Mean Time (Seconds)} \\
\cline{2-3}
& $C = 0.5$ & $C = 0.9$\\
\hline
\hline
ALS& \textbf{0.2517} $\pm$  \textbf{0.0663} \textbf{(180)} \textbf{(180)} & 13.8499 $\pm$  5.8256 \textbf{(106)} \textbf{(20)}\\
\hline
NCG - $\beta^{FR}$& 0.9723 $\pm$  0.3944 \textbf{(180)} \textbf{(180)} & 9.5120 $\pm$  6.7666 \textbf{(94)} \textbf{(20)}\\
\hline
PNCG - $\widetilde{\beta}^{FR}$& 2.4235 $\pm$  1.0916 \textbf{(180)} \textbf{(180)} & 4.4674 $\pm$  1.6256 \textbf{(104)} \textbf{(20)}\\
\hline
PNCG - $\widehat{\beta}^{FR}$& 3.0196 $\pm$  1.8507 \textbf{(179)} \textbf{(180)} & 7.1644 $\pm$  5.4327 \textbf{(106)} \textbf{(20)}\\
\hline
NCG - $\beta^{PR}$& 0.5730 $\pm$  0.3607 \textbf{(180)} \textbf{(180)} & 7.0099 $\pm$  4.3507 \textbf{(83)} \textbf{(20)}\\
\hline
PNCG - $\widetilde{\beta}^{PR}$& 1.7628 $\pm$ 12.6569 \textbf{(180)} \textbf{(180)} & \textbf{2.7751} $\pm$  \textbf{1.9319} \textbf{(109)} \textbf{(20)}\\
\hline
PNCG - $\widehat{\beta}^{PR}$& 1.2049 $\pm$  2.0744 \textbf{(180)} \textbf{(180)} & 4.1549 $\pm$  5.0031 \textbf{(108)} \textbf{(20)}\\
\hline
NCG - $\beta^{HS}$& 0.5285 $\pm$  0.3131 \textbf{(180)} \textbf{(180)} & 6.9515 $\pm$  4.6067 \textbf{(85)} \textbf{(20)}\\
\hline
PNCG - $\widetilde{\beta}^{HS}$& 0.9940 $\pm$  1.1962 \textbf{(180)} \textbf{(180)} & 5.1334 $\pm$  5.6721 \textbf{(107)} \textbf{(20)}\\
\hline
PNCG - $\widehat{\beta}^{HS}$& 1.4841 $\pm$  3.4182 \textbf{(180)} \textbf{(180)} & 5.5534 $\pm$ 12.4827 \textbf{(108)} \textbf{(20)}\\
\hline
\end{tabular}
\end{table}

%% file: Table_50_Paper.tex
\begin{table}[p]   
\centering
\caption{Speed Comparison with $I = 50$.}
\label{Table:SummarizeI50}
\begin{tabular}{| c | l | c | c |}
\hline
& Optimization & \multicolumn{2}{|c|}{Mean Time (Seconds)} \\
\cline{3-4}
&\multicolumn{1}{|c|}{Method}& $C = 0.5$ & $C = 0.9$\\
\hline
\hline
\multirow{4}{*}{$R = 3$}&ALS& \textbf{0.1988} $\pm$  \textbf{0.0368} \textbf{(180)} \textbf{(180)} & 5.1981 $\pm$  0.3444 \textbf{(180)} \textbf{(180)}\\
\cline{2-4}
&NCG - $\beta^{PR}$& 0.7170 $\pm$  0.2830 \textbf{(180)} \textbf{(180)} & 4.4516 $\pm$  1.9664 \textbf{(179)} \textbf{(171)}\\
\cline{2-4}
&PNCG - $\widetilde{\beta}^{PR}$& 0.8335 $\pm$  0.9137 \textbf{(180)} \textbf{(180)} & \textbf{1.6320} $\pm$  \textbf{1.1064} \textbf{(180)} \textbf{(180)}\\
\cline{2-4}
&PNCG - $\widehat{\beta}^{PR}$& 1.1722 $\pm$  1.4899 \textbf{(180)} \textbf{(180)} & 1.6676 $\pm$  0.7855 \textbf{(180)} \textbf{(180)}\\
\hline
\hline
\multirow{4}{*}{$R = 5$}&ALS& \textbf{0.3357} $\pm$  \textbf{0.1509} \textbf{(180)} \textbf{(180)} &10.4698 $\pm$  3.0988 \textbf{(159)} \textbf{(120)}\\
\cline{2-4}
&NCG - $\beta^{PR}$& 1.6522 $\pm$  1.2236 \textbf{(180)} \textbf{(180)} &14.6827 $\pm$ 10.1787 \textbf{(142)} \textbf{(116)}\\
\cline{2-4}
&PNCG - $\widetilde{\beta}^{PR}$& 3.8331 $\pm$ 13.5605 \textbf{(179)} \textbf{(179)} & \textbf{7.4386} $\pm$ \textbf{12.2583} \textbf{(155)} \textbf{(120)}\\
\cline{2-4}
&PNCG - $\widehat{\beta}^{PR}$& 6.1021 $\pm$ 26.0100 \textbf{(179)} \textbf{(180)} &10.4150 $\pm$ 25.0737 \textbf{(156)} \textbf{(120)}\\
\hline
\end{tabular}
\end{table}

%% file: Table_100_Paper.tex
\begin{table}[p]  
\centering
\caption{Speed Comparison with $I = 100$.}
\label{Table:SummarizeI100}
\begin{tabular}{| c | | l | c | c |}
\hline
& Optimization & \multicolumn{2}{|c|}{Mean Time (Seconds)} \\
\cline{3-4}
&\multicolumn{1}{|c|}{Method}& $C = 0.5$ & $C = 0.9$\\
\hline
\hline
\multirow{4}{*}{$R = 3$}&ALS& \textbf{1.9006} $\pm$  \textbf{0.7043} \textbf{(180)} \textbf{(180)} &47.3505 $\pm$  4.3030 \textbf{(180)} \textbf{(180)}\\
\cline{2-4}
&NCG - $\beta^{PR}$&14.3840 $\pm$  6.1019 \textbf{(180)} \textbf{(180)} & 94.9786 $\pm$ 89.6489 \textbf{(180)} \textbf{(180)}\\
\cline{2-4}
&PNCG - $\widetilde{\beta}^{PR}$&15.3848 $\pm$ 24.8887 \textbf{(180)} \textbf{(180)} &\textbf{28.2346 }$\pm$ \textbf{30.9428} \textbf{(180)} \textbf{(180)}\\
\cline{2-4}
&PNCG - $\widehat{\beta}^{PR}$&20.8161 $\pm$ 31.6531 \textbf{(180)} \textbf{(180)} &34.8675 $\pm$ 46.9708 \textbf{(180)} \textbf{(180)}\\
\hline
\hline
\multirow{4}{*}{$R = 5$}&ALS& \textbf{1.9770} $\pm$ \textbf{ 0.4002} \textbf{(180)} \textbf{(180)} &\textbf{57.1086} $\pm$  \textbf{5.5332} \textbf{(180)} \textbf{(179)}\\
\cline{2-4}
&NCG - $\beta^{PR}$&14.8031 $\pm$  6.2776 \textbf{(180)} \textbf{(180)} &124.5449 $\pm$ 95.9350 \textbf{(178)} \textbf{(138)}\\
\cline{2-4}
&PNCG - $\widetilde{\beta}^{PR}$&44.2358 $\pm$ 205.5225 \textbf{(180)} \textbf{(179)} &103.7680 $\pm$ 257.0952 \textbf{(178)} \textbf{(178)}\\
\cline{2-4}
&PNCG - $\widehat{\beta}^{PR}$&66.7177 $\pm$ 157.0857 \textbf{(180)} \textbf{(180)} &151.7887 $\pm$ 356.2924 \textbf{(180)} \textbf{(179)}\\
\hline
\end{tabular}
\end{table}